\documentclass{tac}

\usepackage[colorlinks=true]{hyperref}
\hypersetup{allcolors=[rgb]{0.1,0.1,0.4}}
 \hypersetup{
           breaklinks=true,
        }
\usepackage{breakurl}

\address{Department of Mathematics and Statistics\\
 Masaryk University, Brno, Czech Republic\\
}

\copyrightyear{2022}


\eaddress{kanalas@mail.muni.cz}

\newtheorem{theorem}{Theorem}

\newtheoremrm{rem}{Remark}

\mathrmdef{Hom}
\mathbfdef{Set}

\usepackage{booktabs}
\usepackage{adjustbox}
\usepackage{ltablex,array}
\usepackage{longtable}
\usepackage{makecell}
\usepackage{marvosym}
\usepackage{ltablex,array}
\usepackage{longtable}
\usepackage{makecell}
\usepackage{marvosym}
 
\newcommand\restr[2]{{
\left.
\kern-
\nulldelimiterspace 
#1 
\right|_{#2} 
}}
\usepackage{tikz}
\usepackage{tikz-cd}
\tikzcdset{scale cd/.style={every label/.append style={scale=#1},
    cells={nodes={scale=#1}}}}
\usetikzlibrary{positioning} \usetikzlibrary{decorations.pathreplacing}
\usetikzlibrary{arrows}
\tikzset{>=stealth}    
\usepackage{mathtools}
\usepackage{adjustbox}
\usepackage{quiver}
\usepackage[
backend=biber,
style=alphabetic,
sorting=nyt
]{biblatex}

\begin{document}
\title{A (2,1)-model structure for conceptual completeness}
\author{Kristóf Kanalas}

\maketitle

\begin{abstract}
   We prove the (2,1)-categorical analogue of the small object argument and give a (2,1)-model structure on the category of small coherent categories, coherent functors and natural isomorphisms. It is induced by a higher dimensional example of a reflective factorisation system, determined by the full subcategory of pretoposes. We prove it to be right proper and the generating trivial cofibrations are described. Whitehead's theorem gives conceptual completeness.
\end{abstract}

\tableofcontents

\section{Introduction}

Coherent categories are categories with finite limits, pullback-stable image factorisations and pullback-stable unions. Their importance was established in \cite{makkai} (under the name "logical category") as the class of small coherent categories and coherent functors (the structure-preserving ones) can be identified with many-sorted coherent (also called positive) theories and interpretations, see Section 3 for an overview.

A central result in categorical logic, Makkai's conceptual completeness states that if a coherent functor $F:\mathcal{C}\to \mathcal{D}$ induces an equivalence $F^*:\mathbf{Coh}(\mathcal{D},\mathbf{Set})\to \mathbf{Coh}(\mathcal{C},\mathbf{Set})$ (i.e.~if it is a Morita-equivalence), and $\mathcal{C}$ is a pretopos (a coherent category with disjoint unions and quotients of equivalence relations) then $F$ is an equivalence.

Providing a Quillen model structure is a way of doing axiomatic homotopy theory in categories. It consists of three class of arrows on some bicomplete category: weak equivalences, fibrations and cofibrations, satisfying certain diagrammatic axioms. A number of classical results generalise to this setting, in particular Whitehead's theorem takes the form, that in a model category if $X$ and $Y$ are both bifibrant, a map $f:X\to Y$ is a weak equivalence iff it is a homotopy equivalence.

It is therefore natural to look for a model structure whose weak equivalences are the Morita-equivalences, the bifibrant objects are the pretoposes and between pretoposes two maps are homotopic iff they are naturally isomorphic, as in this case we would get conceptual completeness as an instance of Whitehead's theorem. However the 1-category of small coherent categories and coherent functors is neither complete nor cocomplete, but it is in the (2,1)-categorical sense (as it is proved in \cite{adjoint}). Hence we should look for a (2,1)-model structure with these properties (in the sense of \cite{infmodel}), whose existence will be proved in Section 4. 

Section 2 gives the proof of the (2,1)-categorical small object argument which will be used in the construction. It follows the standard 1-categorical proof (given e.g.~in \cite{hovey}), with naturally modified arguments for the higher categorical setting. It is not necessary for the pure existence of a model structure with the desired properties: since the full subcategory of pretoposes is known to be reflective in the (2,1)-sense, the generalisation of the theory of reflective factorisation systems would also suffice. However our approach has the advantage that it gives an explicit description for the generating cofibrations. Finally we will prove the model structure to be right proper.

I am grateful for the fruitful conversations with John Bourke, Jiří Rosický and P\'al Zs\'amboki.

\section{Small object argument for (2,1)-categories}

From now on let $\mathbf{C}$ denote a strict (2,1)-category, that is, a strict 2-category whose 2-cells are invertible. First we recall the notion of a (co)limit for (2,1)-categories.

\begin{definition}
Given a small 2-diagram $d_\bullet :\mathcal{I}\to \mathbf{C}$ (i.e.~a strict 2-functor of strict (2,1)-categories), its \emph{2-limit} is an object $d$ with a pseudonatural transformation $\pi :\Delta (d)\Rightarrow d_{\bullet }$ such that $\pi _*:\mathbf{C}(a,d)\to Nat(\Delta (a),d_{\bullet })$ is an equivalence of categories (groupoids in this case). (The codomain is the category of pseudonatural transformations from the constant $a$-valued functor to $d_{\bullet }$ with the modifications as the morphisms.)
\end{definition}

\begin{remark}
In elementary terms the 2-limit can be described as a cone

$ $\\
\adjustbox{scale=0.8,center}
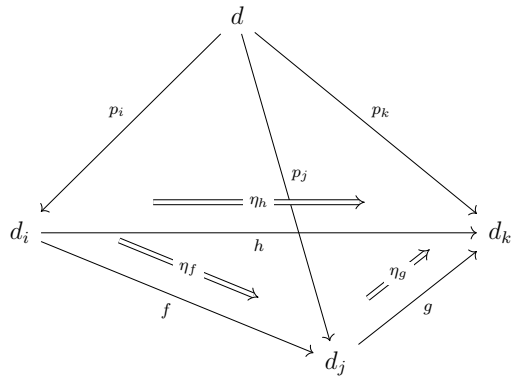
{
\begin{tikzcd}
	&&& d \\
	\\
	\\
	\\
	{d_i} &&&&&& {d_k} \\
	\\
	&&&& {d_j}
	\arrow["{f}"', from=5-1, to=7-5]
	\arrow["{g}"', from=7-5, to=5-7]
	\arrow["{h}"', from=5-1, to=5-7]
	\arrow["{p_i}"', from=1-4, to=5-1]
	\arrow["{p_j}", from=1-4, to=7-5]
	\arrow["{p_k}", from=1-4, to=5-7]
	\arrow["{\eta _f}"{description}, shift left=5, shorten <=34pt, shorten >=34pt, Rightarrow, from=5-1, to=7-5]
	\arrow["{\eta _g}"{description}, shift left=5, shorten <=17pt, shorten >=17pt, Rightarrow, from=7-5, to=5-7]
	\arrow["{\eta _h}"{description}, shift left=5, shorten <=53pt, shorten >=53pt, Rightarrow, from=5-1, to=5-7]
\end{tikzcd}
}
such that for each 2-cell $g\circ f \Rightarrow h$ in the diagram the above tetrahedron commutes. Moreover it has the following universal property: given another such cone $(e,(q_i)_{i\in \mathcal{I}_0},(\nu _f)_{f\in \mathcal{I}_1})$ there is a map $r:e\to d$, unique up to unique natural isomorphism such that there are 2-isomorphisms $\alpha _i : p_i r \Rightarrow q_i$ for which the composite of the 2-cells 

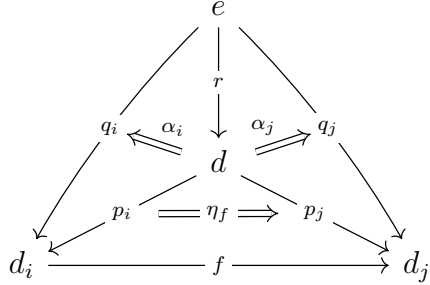
\[\begin{tikzcd}
	&& e \\
	\\
	&& d \\
	{d_i} &&&& {d_j}
	\arrow["r"{description}, from=1-3, to=3-3]
	\arrow[""{name=0, anchor=center, inner sep=0}, "{p_i}"{description}, from=3-3, to=4-1]
	\arrow[""{name=1, anchor=center, inner sep=0}, "{p_j}"{description}, from=3-3, to=4-5]
	\arrow[""{name=2, anchor=center, inner sep=0}, "{q_i}"{description}, curve={height=6pt}, from=1-3, to=4-1]
	\arrow[""{name=3, anchor=center, inner sep=0}, "{q_j}"{description}, curve={height=-6pt}, from=1-3, to=4-5]
	\arrow["f"{description}, from=4-1, to=4-5]
	\arrow["{\eta _f}"{description}, shorten <=13pt, shorten >=13pt, Rightarrow, from=0, to=1]
	\arrow["{\alpha _i}"', shorten <=6pt, shorten >=6pt, Rightarrow, from=3-3, to=2]
	\arrow["{\alpha _j}", shorten <=6pt, shorten >=6pt, Rightarrow, from=3-3, to=3]
\end{tikzcd}\]
is $\nu _f : f q_i \Rightarrow q_j$ (for each arrow $f$ of the diagram).

\end{remark}

In this section we will generalise the small object argument for locally small 2-cocomp\-lete strict (2,1)-categories. The proof follows the one given in \cite{hovey} for the 1-categorical setting.

\begin{definition}
Let $\lambda $ be an ordinal seen as a (2,1)-category with trivial 2-cells. Given a 2-colimit preserving diagram $\lambda \to \mathbf{C}$ with 2-colimit $\mathcal{X}$
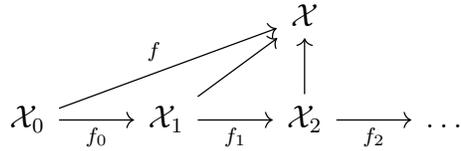
\[\begin{tikzcd}
	&& {\mathcal{X}} \\
	{\mathcal{X}_0} & {\mathcal{X}_1} & {\mathcal{X}_2} & \dots
	\arrow["{f_0}"', from=2-1, to=2-2]
	\arrow["{f_1}"', from=2-2, to=2-3]
	\arrow["{f_2}"', from=2-3, to=2-4]
	\arrow["f", from=2-1, to=1-3]
	\arrow[from=2-2, to=1-3]
	\arrow[from=2-3, to=1-3]
\end{tikzcd}\]
the coprojection map $f:\mathcal{X}_0\to \mathcal{X}$ is called the \emph{transfinite composition} of the $\lambda $-sequence $(f_i)_{i<\lambda }$.
\end{definition}

\begin{definition}
Let $I\subset Arr(\mathbf{C})$ be a set. \emph{$I$-$cell$} is the class of maps that can be written as the transfinite composition of 2-pushouts from $I$. \emph{$I$-$inj$} is the class whose members ($f$) have the following right lifting property: given a square
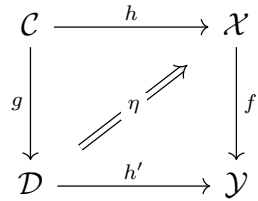
\[\begin{tikzcd}
	{\mathcal{C}} && {\mathcal{X}} \\
	\\
	{\mathcal{D}} && {\mathcal{Y}}
	\arrow["g"', from=1-1, to=3-1]
	\arrow["{h'}", from=3-1, to=3-3]
	\arrow["f", from=1-3, to=3-3]
	\arrow["h", from=1-1, to=1-3]
	\arrow["\eta"{description}, shorten <=11pt, shorten >=11pt, Rightarrow, from=3-1, to=1-3]
\end{tikzcd}\]
with $g\in I$, there is a lifting
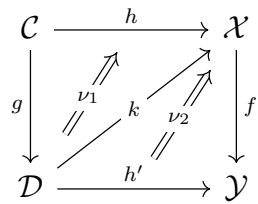
\[\begin{tikzcd}
	{\mathcal{C}} && {\mathcal{X}} \\
	\\
	{\mathcal{D}} && {\mathcal{Y}}
	\arrow["g"', from=1-1, to=3-1]
	\arrow[""{name=0, anchor=center, inner sep=0}, "{h'}", from=3-1, to=3-3]
	\arrow["f", from=1-3, to=3-3]
	\arrow[""{name=1, anchor=center, inner sep=0}, "h", from=1-1, to=1-3]
	\arrow["k"{description}, from=3-1, to=1-3]
	\arrow["{\nu_1}"{description}, shorten <=13pt, shorten >=9pt, Rightarrow, from=3-1, to=1]
	\arrow["{\nu_2}"{description}, shorten <=13pt, shorten >=9pt, Rightarrow, from=0, to=1-3]
\end{tikzcd}\]
such that the pasting of $\nu _1$ and $\nu _2$ is $\eta $, i.e.~$f\nu _1 \circ \nu _2g =\eta$.

\emph{$I$-$proj$} is the class whose members have the left lifting property wrt.~$I$. As usual we set \emph{$I$-$cof$}=($I$-$inj$)-$proj$, and \emph{$I$-$fib$}=($I$-$proj$)-$inj$.
\end{definition}


\begin{proposition}
$I$-cell $\subseteq I$-$cof$.
\end{proposition}

\begin{proof}
Clearly $I\subseteq I$-$cof$, hence it suffices to prove that $I$-$cof$ is closed under pushouts and transfinite compositions. First we show that given
\[\begin{tikzcd}
	\bullet && \bullet && \bullet \\
	\\
	\bullet && \bullet && \bullet
	\arrow[""{name=0, anchor=center, inner sep=0}, "f"', from=1-1, to=3-1]
	\arrow[""{name=1, anchor=center, inner sep=0}, "g", from=1-1, to=1-3]
	\arrow["{f'}"', from=1-3, to=3-3]
	\arrow["{g'}", from=3-1, to=3-3]
	\arrow["h", from=1-3, to=1-5]
	\arrow[""{name=2, anchor=center, inner sep=0}, "m", from=1-5, to=3-5]
	\arrow[""{name=3, anchor=center, inner sep=0}, "k", from=3-3, to=3-5]
	\arrow["\alpha", shorten <=7pt, shorten >=7pt, Rightarrow, from=1, to=0]
	\arrow["\beta", shorten <=7pt, shorten >=7pt, Rightarrow, from=2, to=3]
\end{tikzcd}\]
with the left square being a 2-pushout, there exists a lifting in the right square.

Using that $f$ has the left lifting property wrt.~$m$ we get
\[\begin{tikzcd}
	\bullet && \bullet && \bullet \\
	\\
	\bullet && \bullet && \bullet
	\arrow[""{name=0, anchor=center, inner sep=0}, "f"', from=1-1, to=3-1]
	\arrow["g", from=1-1, to=1-3]
	\arrow["{g'}", from=3-1, to=3-3]
	\arrow["h", from=1-3, to=1-5] 
	\arrow[""{name=1, anchor=center, inner sep=0}, "m", from=1-5, to=3-5]
	\arrow["k", from=3-3, to=3-5]
	\arrow["{l_1}", from=3-1, to=1-5]
	\arrow["\delta"', shorten <=11pt, shorten >=11pt, Rightarrow, from=1, to=3-3]
	\arrow["{\gamma }", shorten <=11pt, shorten >=11pt, Rightarrow, from=1-3, to=0]
\end{tikzcd}\]
such that the pasting of $\alpha $ and $\beta $ is the same as the pasting of $\gamma $ and $\delta $ (i.e.~$k\alpha \circ \beta g =\delta f \circ m\gamma$).

Using the universal property of the 2-pushout we get
\[\begin{tikzcd}
	\bullet && \bullet \\
	\\
	\bullet && \bullet \\
	&&& \bullet
	\arrow[""{name=0, anchor=center, inner sep=0}, "f"', from=1-1, to=3-1]
	\arrow[""{name=1, anchor=center, inner sep=0}, "g", from=1-1, to=1-3]
	\arrow["{f'}", from=1-3, to=3-3]
	\arrow["{g'}"', from=3-1, to=3-3]
	\arrow[""{name=2, anchor=center, inner sep=0}, "{l_1}"', curve={height=12pt}, from=3-1, to=4-4]
	\arrow[""{name=3, anchor=center, inner sep=0}, "h", curve={height=-12pt}, from=1-3, to=4-4]
	\arrow["{l_2}"{description}, dashed, from=3-3, to=4-4]
	\arrow["\alpha", shorten <=7pt, shorten >=7pt, Rightarrow, from=1, to=0]
	\arrow["\nu", shorten >=4pt, Rightarrow, from=3-3, to=2]
	\arrow["\eta", shorten <=4pt, Rightarrow, from=3, to=3-3]
\end{tikzcd}\]
such that the pasting of $\alpha $, $\nu $ and $\eta $ is $\gamma $.

We should prove that $l_2$ is a splitting of $\beta $. It is enough to see that in

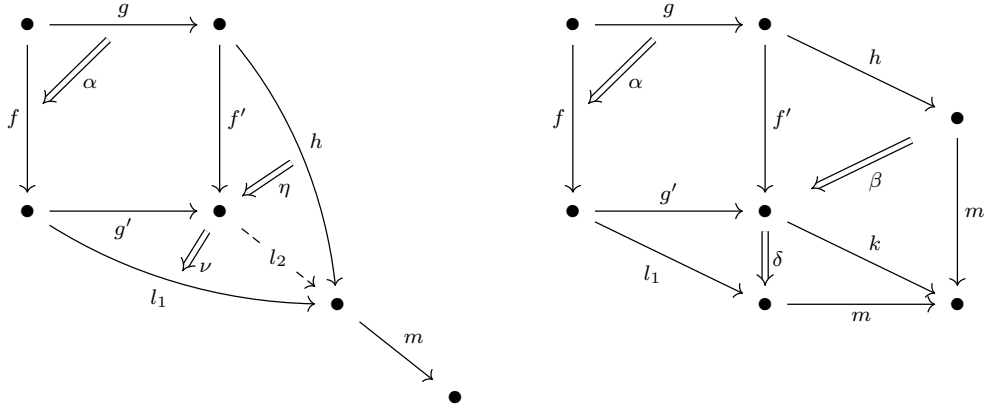
\[\begin{tikzcd}
	\bullet && \bullet &&& \bullet && \bullet \\
	&&&&&&&&& \bullet \\
	\bullet && \bullet &&& \bullet && \bullet \\
	&&& \bullet &&&& \bullet && \bullet \\
	&&&& \bullet
	\arrow[""{name=0, anchor=center, inner sep=0}, "f"', from=1-1, to=3-1]
	\arrow[""{name=1, anchor=center, inner sep=0}, "g", from=1-1, to=1-3]
	\arrow["{f'}", from=1-3, to=3-3]
	\arrow["{g'}"', from=3-1, to=3-3]
	\arrow[""{name=2, anchor=center, inner sep=0}, "{l_1}"', curve={height=12pt}, from=3-1, to=4-4]
	\arrow[""{name=3, anchor=center, inner sep=0}, "h", curve={height=-12pt}, from=1-3, to=4-4]
	\arrow["{l_2}"{description}, dashed, from=3-3, to=4-4]
	\arrow["m", from=4-4, to=5-5]
	\arrow["{g'}", from=3-6, to=3-8]
	\arrow["{f'}", from=1-8, to=3-8]
	\arrow[""{name=4, anchor=center, inner sep=0}, "f"', from=1-6, to=3-6]
	\arrow[""{name=5, anchor=center, inner sep=0}, "g", from=1-6, to=1-8]
	\arrow["k", from=3-8, to=4-10]
	\arrow["h", from=1-8, to=2-10]
	\arrow["m", from=2-10, to=4-10]
	\arrow["{l_1}"', from=3-6, to=4-8]
	\arrow["m"', from=4-8, to=4-10]
	\arrow["\beta", shorten <=10pt, shorten >=10pt, Rightarrow, from=2-10, to=3-8]
	\arrow["\delta", Rightarrow, from=3-8, to=4-8]
	\arrow["\alpha", shorten <=7pt, shorten >=7pt, Rightarrow, from=1, to=0]
	\arrow["\nu", shorten >=4pt, Rightarrow, from=3-3, to=2]
	\arrow["\eta", shorten <=4pt, Rightarrow, from=3, to=3-3]
	\arrow["\alpha", shorten <=7pt, shorten >=7pt, Rightarrow, from=5, to=4]
\end{tikzcd}\]
the 2-cells filling the boundaries are identical as in this case both $ml_2$ and $k$ are suitable splittings, hence there is a unique natural isomorphism $\varepsilon:ml_2 \Rightarrow k$ for which $\varepsilon f' \circ m\eta =\beta $ and $m\nu \circ \varepsilon ^{-1}g' =\delta$ (and the first equality means that $l_2$ is a splitting). The equality of the 2-cells follows from the identities observed above.

Now assume that each $f_i$ ($i<\lambda $) has left lifting property wrt.~$m$ (and that $f_i$-s form a (co)continuous sequence). We have to prove that its transfinite composition $f$ has the same lifting property. We have a lift in 
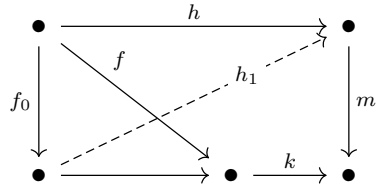
\[\begin{tikzcd}
	\bullet &&& \bullet \\
	\\
	\bullet && \bullet & \bullet
	\arrow["{f_0}"', from=1-1, to=3-1]
	\arrow["f"{pos=0.3}, from=1-1, to=3-3]
	\arrow[from=3-1, to=3-3]
	\arrow["h", from=1-1, to=1-4]
	\arrow["k", from=3-3, to=3-4]
	\arrow["m", from=1-4, to=3-4]
	\arrow["{h_1}"{description, pos=0.7}, dashed, from=3-1, to=1-4]
\end{tikzcd}\]
and this way one defines $h_i$ for each successor ordinal $i<\lambda $. At limit steps $h_i$ is induced by the universal property of the 2-colimit of the sequence. Compatibility of the 2-cells is checked as before.
\end{proof}

It is worth to spell out explicitly:

\begin{proposition}
Left lifting properties are preserved by 2-pushouts and transfinite compositions. Dually, right lifting properties are preserved by 2-pullbacks and transfinite cocompositions (2-limit of the reversed sequence). In particular (taking $\lambda =2$) $I$-$inj$ and $I$-$proj$ are subcategories.
 
\label{lifting}
\end{proposition}

\begin{proposition}
$I$-cell is closed under transfinite composition.
\end{proposition}

\begin{proof}
We need to prove that "the transfinite composition of transfinite compositions is a transfinite composition", i.e.~that if we have a sequential (2-)diagram then its colimit can be computed as the colimit of any cofinal subsequence. This is Proposition 4.1.1.8. in \cite{lurietopos}.
\end{proof}

\begin{proposition}
The homotopy pushout of a coproduct of maps from $I$ is in $I$-cell.
\end{proposition}

\begin{proof}
Let $g_j$ $(j\in J)$ be a family of arrows from $I$. Their coproduct is the induced map:
\[\begin{tikzcd}
	{\mathcal{C}_j} && {\cup _{j\in J} \mathcal{C}_j} \\
	\\
	{\mathcal{D}_j} && {\cup _{j\in J}\mathcal{D}_j}
	\arrow["{g_j}", from=1-1, to=3-1]
	\arrow[from=3-1, to=3-3]
	\arrow[from=1-1, to=1-3]
	\arrow["{\cup g_j}", dashed, from=1-3, to=3-3]
	\arrow["{\chi _j}"{description}, shorten <=11pt, shorten >=11pt, Rightarrow, from=1-3, to=3-1]
\end{tikzcd}\]
Now take the 2-pushout:
\[\begin{tikzcd}
	{\cup _j\mathcal{C}_j} && {\mathcal{X}} \\
	\\
	{\cup _j\mathcal{D}_j} && {\mathcal{Y}}
	\arrow["{\cup g_j}"', from=1-1, to=3-1]
	\arrow["{h_1}", from=3-1, to=3-3]
	\arrow["f", from=1-3, to=3-3]
	\arrow["{h_0}", from=1-1, to=1-3]
	\arrow["\eta"{description}, shorten <=11pt, shorten >=11pt, Rightarrow, from=1-3, to=3-1]
\end{tikzcd}\]
We will proceed by transfinite recursion and take: $X_0=X$, $\rho _0=f$ and $i_{0,0}=1_X$. In the successor step we form the 2-pushout of $g_j:\mathcal{C}_j\to \mathcal{D}_j$ and $\mathcal{C}_j\to \cup \mathcal{C}_j \xrightarrow{h_0} \mathcal{X} \xrightarrow{i_{0,j}}\mathcal{X}_j$ to get $X_{j+1}$ and induce $\rho _{j+1}$ by the universal property of the square. Hence we get a commutative cube (where the faces are filled with the obvious 2-cells):
\[\begin{tikzcd}
	{\mathcal{C}_j} &&&& {\mathcal{X}_j} \\
	& {\cup \mathcal{C}_j} & {\mathcal{X}} & {\mathcal{X}_j} \\
	\\
	& {\cup \mathcal{D}_j} && {\mathcal{Y}} \\
	{\mathcal{D}_j} &&&& {\mathcal{X}_{j+1}}
	\arrow["{h_0}", from=2-2, to=2-3]
	\arrow["{i_{0,j}}", from=2-3, to=2-4]
	\arrow["{\rho _j}"{description}, from=2-4, to=4-4]
	\arrow["{\cup g_j}"{description}, from=2-2, to=4-2]
	\arrow["{h_1}"{description}, from=4-2, to=4-4]
	\arrow[from=1-1, to=2-2]
	\arrow[Rightarrow, no head, from=1-5, to=2-4]
	\arrow[from=5-1, to=4-2]
	\arrow[dashed, from=5-5, to=4-4]
	\arrow[from=1-5, to=5-5]
	\arrow[from=1-1, to=1-5]
	\arrow["{g_j}"{description}, from=1-1, to=5-1]
	\arrow[from=5-1, to=5-5]
	\arrow["f"{description}, from=2-3, to=4-4]
\end{tikzcd}\]
When $j$ is a limit ordinal $\mathcal{X}_j$ is given by the transfinite composition
\[\begin{tikzcd}
	&&& {\mathcal{Y}} \\
	&&& {\mathcal{X}_j} \\
	{\mathcal{X}_0} && {\mathcal{X}_1} && \dots
	\arrow["{i_{0,1}}"', from=3-1, to=3-3]
	\arrow["{i_{1,2}}"', from=3-3, to=3-5]
	\arrow["{i_{0,j}}"{description}, from=3-1, to=2-4]
	\arrow["{i_{1,j}}"'{pos=0.7}, from=3-3, to=2-4]
	\arrow[dashed, from=2-4, to=1-4]
	\arrow["{\rho_1}"{description}, curve={height=-6pt}, from=3-3, to=1-4]
	\arrow["{\rho _0}"{description}, curve={height=-6pt}, from=3-1, to=1-4]
\end{tikzcd}\]
(the 3-cells are filled). We claim that with $\lambda =|J|$ the map $\mathcal{X}\to \mathcal{X}_{\lambda}$ is also a homotopy pushout for $\cup g_j$ along $h_0$. To see this we should find some 2-cells for

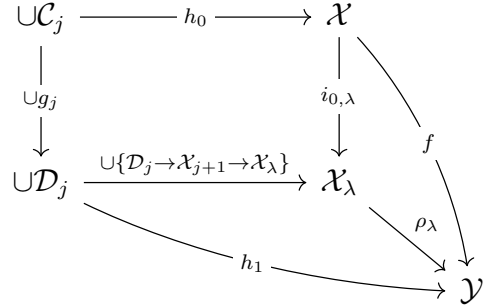
\[\begin{tikzcd}
	{\cup \mathcal{C}_j} &&& {\mathcal{X}} \\
	\\
	{\cup \mathcal{D}_j} &&& {\mathcal{X}_\lambda} \\
	&&&& {\mathcal{Y}}
	\arrow["{h_0}"{description}, from=1-1, to=1-4]
	\arrow["{i_{0,\lambda}}"{description}, from=1-4, to=3-4]
	\arrow["{\cup g_j}"{description}, from=1-1, to=3-1]
	\arrow["{\small{\cup \{\mathcal{D}_j\to \mathcal{X}_{j+1}\to \mathcal{X}_\lambda\}}}", from=3-1, to=3-4]
	\arrow["{\rho _{\lambda}}", from=3-4, to=4-5]
	\arrow["f"{description}, curve={height=-12pt}, from=1-4, to=4-5]
	\arrow["{h_1}"{description}, curve={height=12pt}, from=3-1, to=4-5]
\end{tikzcd}\]
whose composite is $\eta$. They can be found on the surface of the commutative 3-simplicial set

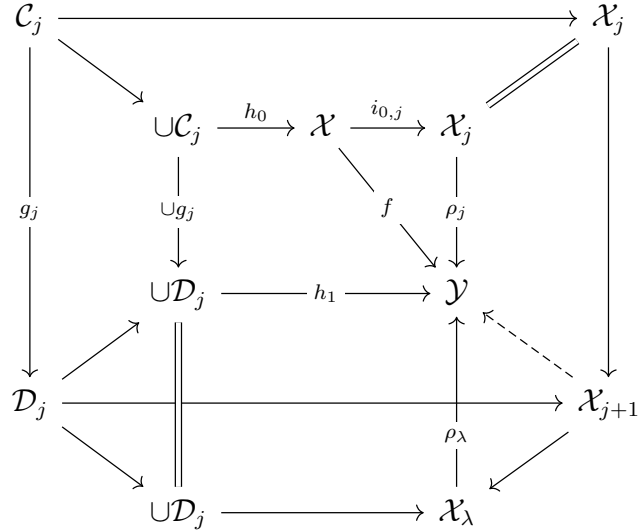
\[\begin{tikzcd}
	{\mathcal{C}_j} &&&& {\mathcal{X}_j} \\
	& {\cup \mathcal{C}_j} & {\mathcal{X}} & {\mathcal{X}_j} \\
	\\
	& {\cup \mathcal{D}_j} && {\mathcal{Y}} \\
	{\mathcal{D}_j} &&&& {\mathcal{X}_{j+1}} \\
	& {\cup\mathcal{D}_j} && {\mathcal{X}_{\lambda}}
	\arrow["{h_0}", from=2-2, to=2-3]
	\arrow["{i_{0,j}}", from=2-3, to=2-4]
	\arrow["{\rho _j}"{description}, from=2-4, to=4-4]
	\arrow["{\cup g_j}"{description}, from=2-2, to=4-2]
	\arrow["{h_1}"{description}, from=4-2, to=4-4]
	\arrow[from=1-1, to=2-2]
	\arrow[Rightarrow, no head, from=1-5, to=2-4]
	\arrow[from=5-1, to=4-2]
	\arrow[dashed, from=5-5, to=4-4]
	\arrow[from=1-5, to=5-5]
	\arrow[from=1-1, to=1-5]
	\arrow["{g_j}"{description}, from=1-1, to=5-1]
	\arrow[from=5-1, to=5-5]
	\arrow["f"{description}, from=2-3, to=4-4]
	\arrow[Rightarrow, no head, from=4-2, to=6-2]
	\arrow[from=5-1, to=6-2]
	\arrow["{\rho _\lambda}"{description, pos=0.3}, from=6-4, to=4-4]
	\arrow[from=5-5, to=6-4]
	\arrow[from=6-2, to=6-4]
\end{tikzcd}\]
\end{proof}

\begin{definition}
An object $\mathcal{X}$ of $\mathbf{C}$ is \emph{$\lambda $-small} wrt.~a subcategory $J$ if $\mathbf{C}(\mathcal{X},-)$ commutes with $\lambda $-filtered sequential 2-colimits formed in $J$. $\mathcal{X}$ is small if it is $\lambda $-small for some $\lambda $.
\end{definition}

\begin{theorem}[Small object argument]
Let $I\subset Arr(\mathbf{C})$ be a set, and assume that domains of $I$ are small relative to $I$-cell. Then for any map $f:\mathcal{X}\to \mathcal{Y}$ there are arrows $\mathcal{X}\xrightarrow{f'}\mathcal{Z}\xrightarrow{f''}\mathcal{Y}$ such that $f'\in I$-cell, $f''\in I$-$inj$ and $f'' \circ f'$ is isomorphic to $f$.
\label{smallob}
\end{theorem}

\begin{proof}
We proceed by transfinite recursion and take $\mathcal{Z}_0=\mathcal{X}$, $\rho _0 =f$ and $i_{0,0}=1_{\mathcal{X}}$.

For a successor ordinal $j+1$ collect all squares

\[\begin{tikzcd}
	{\mathcal{A}_s} && {\mathcal{Z}_j} \\
	\\
	{\mathcal{B}_s} && {\mathcal{Y}}
	\arrow["{g_s}"{description}, from=1-1, to=3-1]
	\arrow["{h_s}"{description}, from=1-1, to=1-3]
	\arrow["{\rho _j}"{description}, from=1-3, to=3-3]
	\arrow["{k_s}"{description}, from=3-1, to=3-3]
	\arrow["{\eta _s}"{description}, shorten <=11pt, shorten >=11pt, Rightarrow, from=1-3, to=3-1]
\end{tikzcd}\]
with $g_s\in I$ to an $S$-indexed set, then form the 2-pushout of $\sqcup g_s$ and $\sqcup h_s$ and induce $\rho _{j+1}$:

\[\begin{tikzcd}
	{\sqcup \mathcal{A}_s} && {\mathcal{Z}_j} \\
	\\
	{\sqcup \mathcal{B}_s} && {\mathcal{Z}_{j+1}} \\
	&&& {\mathcal{Y}}
	\arrow["{\sqcup h_s}", from=1-1, to=1-3]
	\arrow["{\sqcup g_s}"', from=1-1, to=3-1]
	\arrow[from=3-1, to=3-3]
	\arrow["{i_{j,j+1}}"', from=1-3, to=3-3]
	\arrow[""{name=0, anchor=center, inner sep=0}, "{\rho _j}", curve={height=-12pt}, from=1-3, to=4-4]
	\arrow[""{name=1, anchor=center, inner sep=0}, "{\sqcup k_s}"', curve={height=12pt}, from=3-1, to=4-4]
	\arrow["{\rho _{j+1}}", dashed, from=3-3, to=4-4]
	\arrow[shorten <=17pt, shorten >=17pt, Rightarrow, from=1-3, to=3-1]
	\arrow[shorten <=4pt, shorten >=2pt, Rightarrow, from=0, to=3-3]
	\arrow[shorten <=2pt, shorten >=3pt, Rightarrow, from=3-3, to=1]
\end{tikzcd}\]
Note that the composition of the three 2-cells is the natural isomorphism induced by $\{ \eta _s : s\in S \}$. (*)

When $j$ is a limit ordinal we form the transfinite composition

\[\begin{tikzcd}
	&&&& {\mathcal{Y}} \\
	&&&& {\mathcal{Z}_j} \\
	{\mathcal{Z}_0} && {\mathcal{Z}_1} && \dots
	\arrow["{i_{0,1}}"{description}, from=3-1, to=3-3]
	\arrow["{i_{1,2}}"{description}, from=3-3, to=3-5]
	\arrow["{i_{0,j}}"{description, pos=0.4}, from=3-1, to=2-5]
	\arrow["{i_{1,j}}"{description, pos=0.4}, from=3-3, to=2-5]
	\arrow["{\rho _0}"{description}, curve={height=-12pt}, from=3-1, to=1-5]
	\arrow["{\rho_1}"{description}, curve={height=-6pt}, from=3-3, to=1-5]
	\arrow["{\rho _j}"', dashed, from=2-5, to=1-5]
\end{tikzcd}\]

Let $\lambda $ be a cardinal, such that domains of $I$ are $\lambda $-small. The composition $\mathcal{X} \xrightarrow{i_{0,\lambda}} \mathcal{Z}_{\lambda} \xrightarrow{\rho _{\lambda}} \mathcal{Y}$ is isomorphic to $f$ and $i_{0,\lambda} \in I$-cell by the previous propositions.
 
It remains to prove that $\rho _\lambda \in I$-$inj$. Take a square
\[\begin{tikzcd}
	{\mathcal{A}} && {\mathcal{Z}_\lambda} \\
	\\
	{\mathcal{B}} && {\mathcal{Y}}
	\arrow["h"{description}, from=1-1, to=1-3]
	\arrow["{\rho _\lambda}"{description}, from=1-3, to=3-3]
	\arrow["g"{description}, from=1-1, to=3-1]
	\arrow["k"{description}, from=3-1, to=3-3]
	\arrow["\eta"{description}, shorten <=11pt, shorten >=11pt, Rightarrow, from=1-3, to=3-1]
\end{tikzcd}\]
As $\mathcal{A}$ is $\lambda $-small, $h$ factors through some stage $\mathcal{Z}_j$ (up to isomorphism). This means that the back face of the left cube in 

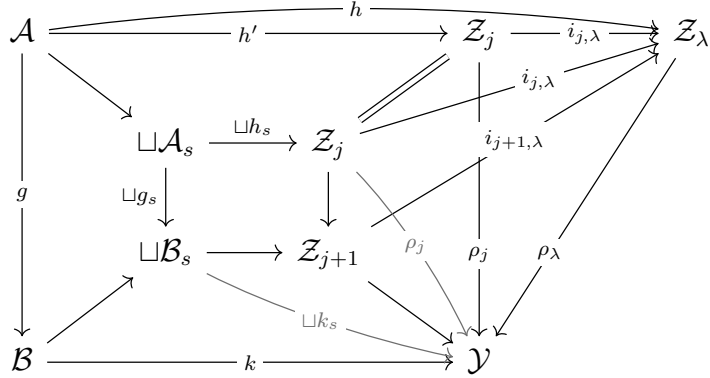
\[\begin{tikzcd}
	{\mathcal{A}} &&& {\mathcal{Z}_{j}} && {\mathcal{Z}_{\lambda}} \\
	& {\sqcup \mathcal{A}_s} & {\mathcal{Z}_{j}} \\
	& {\sqcup \mathcal{B}_s} & {\mathcal{Z}_{j+1}} \\
	{\mathcal{B}} &&& {\mathcal{Y}}
	\arrow["{h'}"{description}, from=1-1, to=1-4]
	\arrow["{\rho _j}"{description, pos=0.7}, from=1-4, to=4-4]
	\arrow["g"{description}, from=1-1, to=4-1]
	\arrow["k"{description}, from=4-1, to=4-4]
	\arrow["{\sqcup g_s}"', from=2-2, to=3-2]
	\arrow[from=3-2, to=3-3]
	\arrow[from=2-3, to=3-3]
	\arrow["{\sqcup h_s}", from=2-2, to=2-3]
	\arrow[from=1-1, to=2-2]
	\arrow[from=4-1, to=3-2]
	\arrow[from=3-3, to=4-4]
	\arrow[Rightarrow, no head, from=2-3, to=1-4]
	\arrow["{i_{j,\lambda }}"{description}, from=1-4, to=1-6]
	\arrow["{i_{j,\lambda }}"{description, pos=0.6}, from=2-3, to=1-6]
	\arrow["{i_{j+1,\lambda }}"{description}, from=3-3, to=1-6]
	\arrow["{\rho _{\lambda}}"{description, pos=0.7}, from=1-6, to=4-4]
	\arrow["{\rho _j}"{description}, color={rgb,255:red,110;green,110;blue,110}, curve={height=-6pt}, from=2-3, to=4-4]
	\arrow["{\sqcup k_s}"{description}, color={rgb,255:red,110;green,110;blue,110}, curve={height=6pt}, from=3-2, to=4-4]
	\arrow["h"{description}, curve={height=-12pt}, from=1-1, to=1-6]
\end{tikzcd}\]
was considered in the formation of $\mathcal{Z}_{j+1}$.
This face is just the gluing of
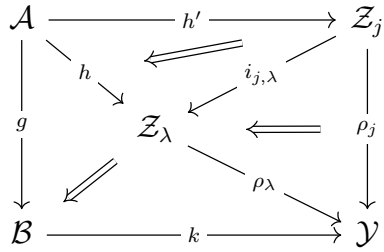
\[\begin{tikzcd}
	{\mathcal{A}} &&& {\mathcal{Z}_{j}} \\
	& {\mathcal{Z}_{\lambda}} \\
	{\mathcal{B}} &&& {\mathcal{Y}}
	\arrow["{h'}"{description}, from=1-1, to=1-4]
	\arrow["g"{description}, from=1-1, to=3-1]
	\arrow["k"{description}, from=3-1, to=3-4]
	\arrow[""{name=0, anchor=center, inner sep=0}, "{\rho _j}"{description}, from=1-4, to=3-4]
	\arrow[""{name=1, anchor=center, inner sep=0}, "h"{description}, from=1-1, to=2-2]
	\arrow["{i_{j,\lambda}}"{description}, from=1-4, to=2-2]
	\arrow["{\rho _\lambda}"{description}, from=2-2, to=3-4]
	\arrow[shorten <=4pt, shorten >=8pt, Rightarrow, from=2-2, to=3-1]
	\arrow[shorten <=17pt, shorten >=22pt, Rightarrow, from=0, to=2-2]
	\arrow[shorten <=36pt, shorten >=22pt, Rightarrow, from=1-4, to=1]
\end{tikzcd}\]

By (*) the left cube is a commutative (identical) 3-cell, and so is the cone over the $\mathcal{Z}_n$'s. Hence the lift $\mathcal{B}\to \sqcup \mathcal{B}_s \to \mathcal{Z}_{j+1} \to \mathcal{Z}_\lambda $ is a splitting of $\eta $.
\end{proof}

Finally we give the definition of a (2,1)-model structure. This is the special case of Definition 1.1 in \cite{infmodel} except that we require the existence of all limits and colimits not just the finite ones.

\begin{definition}
Let $\mathbf{C}$ be a 2-complete and 2-cocomplete (2,1)-category. Given three sub-2-categories $W$ (called weak equivalences), $Fib$ (called fibrations) and $Cof$ (called cofibrations) we say that $(\mathbf{C},W,Fib,Cof)$ is a \emph{model category} if the following axioms are satisfied:
\begin{itemize}
    \item Given $h\cong gf$ if two maps are weak equivalences then so is the third one.
    \item All three subcategories are closed under retracts, i.e.~given a diagram
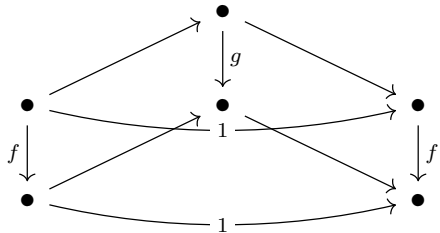
\[\begin{tikzcd}
	&& \bullet \\
	\bullet && \bullet && \bullet \\
	\bullet &&&& \bullet
	\arrow[from=2-1, to=1-3]
	\arrow[from=1-3, to=2-5]
	\arrow[from=3-1, to=2-3]
	\arrow[from=2-3, to=3-5]
	\arrow["f", from=2-5, to=3-5]
	\arrow["g", from=1-3, to=2-3]
	\arrow["f"', from=2-1, to=3-1]
	\arrow["1"{description}, curve={height=12pt}, from=2-1, to=2-5]
	\arrow["1"{description}, curve={height=12pt}, from=3-1, to=3-5]
\end{tikzcd}\]
with strictly commuting front face, filled with the identity 3-cell, $g\in W/Fib/Cof$ implies $f\in W/Fib/Cof$.
\item Maps in $Fib $ have the right lifting property against $W\cap Cof$ and maps in $W\cap Fib$ have the right lifting property against $Cof$.
\item Every map $h$ can be written as $h\cong gf \cong g'f'$ with $f\in Cof $, $g\in W\cap Fib$ and $f'\in W\cap Cof$, $g'\in Fib$.
\end{itemize}
\end{definition}

\section{An overview of coherent categories}

\begin{definition}
A category $\mathcal{C}$ is \emph{coherent}, if it
\begin{itemize}
\item has finite limits,
\item has pullback-stable images, i.e.~every arrow can be factored as an effective epimorphism followed by a monomorphism, and effective epimorphisms are stable under pullbacks,
\item has pullback-stable finite unions.
\end{itemize}

A functor $F:\mathcal{C}\to \mathcal{D}$ is \emph{coherent} if it preserves finite limits, effective epimorphisms and finite unions.
\end{definition}

We will denote the (2,2)-category of small coherent categories, coherent functors and all natural transformations by $\mathbf{Coh}$ and the corresponding (2,1)-category whose 2-cells are the natural isomorphisms by $\mathbf{Coh}_{\sim }$. We will make use of the following results:

\begin{theorem}
The 1-category $\mathbf{Coh}_1$ is accessible.
\end{theorem}

\begin{proof}
By \cite{sketches} $\mathbf{Coh}_1$ is a finitary injectivity class in a presheaf category, hence it is accessible.
\end{proof}

It is well-known since \cite{flexible} that $\mathbf{Coh}_{\sim }$ is 2-complete. Using accessibility of $\mathbf{Coh}_1$, it follows by Section 9.3 of \cite{adjoint} that $\mathbf{Coh}_{ \sim}$ is 2-cocomplete as well, so giving:

\begin{theorem}
$\mathbf{Coh}_{\sim }$ is 2-complete and 2-cocomplete.
\end{theorem}

\begin{theorem}
Let $d_{\bullet }:\mathcal{\lambda }\to \mathbf{Coh}_{\sim }$ be a strict diagram (i.e.~all 2-cells are identical). Then its colimit $d$ in $\mathbf{Coh}_1$ is also the 2-colimit in $\mathbf{Coh}_{\sim }$.
\label{nicefiltered}
\end{theorem}

\begin{proof}
We must show that $\mathbf{Coh}_{\sim }(d,a)\to Strict(d_{\bullet },\Delta(a))\to Nat(d_{\bullet},\Delta(a))$ is an equivalence, or equivalently that the second map is an equivalence since the first one is an iso, or even just that the second map is essentially surjective since it is automatically fully faithful.  Here $Strict$ refers to strict cocones.  Consider a pseudo-natural transformation $d_{\bullet}\Rightarrow \Delta(a)$. Since filtered colimits are 2-colimits in $\mathbf{Cat}$ by Lemma 5.4.9 of \cite{accessible}, we can replace this by an isomorphic strict cocone in $\mathbf{Cat}$ --- but since being a coherent functor is isomorphism invariant, this means that the strict cocone belongs to $\mathbf{Coh}_{\sim }$ as well, proving the claim.
\end{proof}

This implies that in the inductive proof of Theorem \ref{smallob} the sequence $\mathcal{Z}_0 \xrightarrow{i_{0,1}} \dots $ can be chosen to be strict, so any $\lambda $ with $cf(\lambda )> sup \{\kappa _f :f\in I\}$ works, where $\kappa _f$ is the presentability rank of $dom(f)$. Hence we have:

\begin{theorem}
Let $I$ be a small set of coherent functors. Given a coherent functor $\mathcal{C}\xrightarrow{M} \mathcal{E}$ it is isomorphic to a composition $\mathcal{C}\xrightarrow{F}\mathcal{D}\xrightarrow{G}\mathcal{E}$ where $F\in I$-cell (moreover $F$ is the strict transfinite composition of 2-pushouts from $I$) and $G\in I$-$inj$. In particular $F\in I$-$cof$.
\label{smallforcoh}
\end{theorem}

Now we recall from \cite{makkai} the correspondence between coherent categories/coherent functors and coherent theories/interpretations.

\begin{definition}
Let $L $ be a (many-sorted) signature. An $L $-formula is called \emph{coherent} if it is built up from atomic formulas using finite $\wedge $ (including $\top$), finite $\vee$ (including $\bot $) and $\exists $. Their class is denoted by $L_{\omega \omega }^g$.

A formula of the form $\forall x_1 \dots \forall x_n (\varphi \to \psi )$ with $\varphi ,\psi \in L_{\omega \omega }^g$ is called a \emph{coherent sequent} and it is written as $\varphi \Rightarrow \psi $. A \emph{coherent theory} is a set of coherent sequents.
\end{definition}

\begin{definition}
Given a signature $L$, an \emph{$L$-structure} in a category $\mathcal{C}$ with finite products, associates to each sort $X$ an object $M(X)$ of $\mathcal{C}$, to each relation symbol $R$ a subobject $M(R) \leq M(X_1)\times \dots \times M(X_n)$, and to a function symbol $f$ a morphism $M(f):M(X_1)\times \dots \times M(X_n)\to M(X)$. 
\end{definition}

To define the models of a coherent theory internally to some coherent category $\mathcal{C}$ we need to interpret $L_{\omega \omega }^g$ formulas in an $L$-structure $M$ in $\mathcal{C}$:

\begin{definition}
Let $M$ be an $L$-structure in a coherent category $\mathcal{C}$. The \emph{interpretation} of a coherent formula is given by the following steps:
\begin{itemize}
\item If $\vec{x}=(x_1,\dots x_n)$ is a finite sequence of free variables, $x_i$ is of sort $X_i$, then $M(\vec{x})=M(X_1)\times \dots \times M(X_n)$.
\item If $t$ is a term (of sort $Y$) whose free variables are from $\vec{x}$, then $M_{\vec{x}}(t)$ will be an arrow $M(\vec{x})\to M(Y)$ in the following way:
\begin{itemize}
\item If $t=x_i$, then $M_{\vec{x}}(t)$ is the projection map $M(\vec{x})\to M(X_i)$.
\item If $t=f(t_1,\dots t_n)$, then $M_{\vec{x}}(t)$ is the composite $M(\vec{x})\xrightarrow {\langle M_{\vec{x}}(t_1),\dots \rangle} \prod M(Y_i) \xrightarrow{M(f)} M(Y)$
\end{itemize}
When $\mathcal{C}=\mathbf{Set}$, these are the functions which to a possible evaluation of $\vec{x}$ assign the induced value of $t$.
\item If $\varphi $ is a formula, whose free variables are among $\vec{x}=(x_1,\dots x_n)$, then its interpretation in the context $\vec{x}$ is a subobject $M_{\vec{x}}(\varphi )\leq M(\vec{x})$. It can be readily checked that in the case of $\mathbf{Set}$-models this gives precisely the set of evaluations of $\vec{x}$ which make $\varphi $ valid in $M$.

\begin{tabularx}{\textwidth}{XX} \begin{tikzpicture}
\node (1) at (0,0) {$M_{\vec{x}}(t_1\approx t_2) $};
\node[anchor=base,right=10mm of 1] (2) {$M(\vec{x})$};
\node[anchor=base,right=10mm of 2] (3) {$M(Y)$};
\draw[right hook ->] (1)--(2) node [midway,above] {\small{$e$}};
\draw[postaction={transform canvas={yshift=-1mm},draw}] [->] (2) -- (3) node [midway,above] {\small{$M_{\vec{x}}(t_1)$}} node [midway,below] {\small{$M_{\vec{x}}(t_2)$}};
\end{tikzpicture} &
is an equalizer.
\end{tabularx}

\begin{tabularx}{\textwidth}{XX}
\begin{tikzpicture}
\node (3) at (0,0) {$M_{\vec{x}}(R(t_1,\dots t_n))$};
\node[anchor=base,right=10mm of 3] (2a) {$M(\vec{x})$};
\node[anchor=base,below=10mm of 3] (2b) {$M(R)$};
\node[anchor=base,below=10mm of 2a] (1) {$\prod _{i=1}^n M(Y_i)$};
\draw[right hook->] (3)--(2a);
\draw[->] (3)--(2b);
\draw[->] (2a)--(1) node [midway,right] {$\langle M_{\vec{x}}(t_1),\dots \rangle$};
\draw[right hook->] (2b)--(1) node [midway,below] {$M(i)$};
\end{tikzpicture}
& \centering is a pullback.
\end{tabularx}
\begin{tabularx}{\textwidth}{X}
 $M_{\vec{x}}(\bigwedge \Theta)=\bigwedge\{M_{\vec{x}}(\theta):\theta\in \Theta\}$ \\
 \end{tabularx}
 \begin{tabularx}{\textwidth}{X}
 $M_{\vec{x}}(\bigvee \Theta)=\bigvee\{M_{\vec{x}}(\theta):\theta\in \Theta\}$ \\
 \end{tabularx}
 \begin{tabularx}{\textwidth}{X}
\makecell{ $M_{\vec{x}}(\exists y \varphi)$ (where $y$ is not in $\vec{x}$) is the eff.~epi-mono factorisation:
\\
\begin{tikzpicture}
\node (3) at (0,0) {$M_{\vec{x},y}(\varphi)$};
\node[anchor=base,right=10mm of 3] (2a) {$M(\vec{x},y)$};
\node[anchor=base,right=10mm of 2a] (2b) {$M(\vec{x})$};
\node[anchor=base,below=10mm of 2a] (1) {$M_{\vec{x}}(\exists y \varphi)$};
\draw[right hook->] (3)--(2a);
\draw[->] (2a)--(2b) node [midway,above] {$\pi _{\vec{x}}$};
\draw[->>] (3)--(1);
\draw[right hook->] (1)--(2b);
\end{tikzpicture}
}
\end{tabularx}

\end{itemize}
\label{intpret}
\end{definition}

\begin{definition}
The sequent $\varphi \Rightarrow \psi$ is \emph{valid} in the structure $M$ (in symbols: $M\models \varphi \Rightarrow \psi $), iff $ M_{\vec{x}}(\varphi ) \leq M_{\vec{x}}(\psi ) $ (where $\vec{x}$ is the collection of all free variables in $\varphi \Rightarrow \psi $).

$M$ is a \emph{model} of the theory $T$ if all the sequents from $T$ are valid in $M$.

A \emph{homomorphism} $\alpha :M\to M'$ of $T$-models consists of an arrow $\alpha _X: M(X)\to M'(X)$ for each sort $X$, for which the square

\[\begin{tikzcd}
	{M(X_1)\times \dots \times M(X_n)} && {M(Y)} \\
	\\
	{M'(X_1)\times \dots \times M'(X_n)} && {M'(Y)}
	\arrow["{M(f)}", from=1-1, to=1-3]
	\arrow["{\alpha _{X_1} \times \dots \times \alpha _{X_n}}"', from=1-1, to=3-1]
	\arrow["{\alpha _Y}", from=1-3, to=3-3]
	\arrow["{M'(f)}", from=3-1, to=3-3]
\end{tikzcd}\]
commutes and the dashed arrow in

\[\begin{tikzcd}
	{M(R)} && {M(X_1)\times \dots \times M(X_n)} \\
	\\
	{M'(R)} && {M'(X_1)\times \dots \times M'(X_n)}
	\arrow[hook, from=1-1, to=1-3]
	\arrow["{\alpha _{X_1} \times \dots \times \alpha _{X_n}}", from=1-3, to=3-3]
	\arrow[hook, from=3-1, to=3-3]
	\arrow[dashed, from=1-1, to=3-1]
\end{tikzcd}\]
exists.

The category of $T$-models and homomorphisms in a category $\mathcal{C}$ is denoted by $T$-$mod(\mathcal{C})$.
\label{modelincat}
\end{definition}

At first we will replace categories with theories:

\begin{definition}
The \emph{canonical language} of the category $\mathcal{C}$ has the signature $L=L_{\mathcal{C}}$ which contains a sort $\bar{A}$ for every object $A$ of $\mathcal{C}$, and a function symbol $\bar{f}:\bar{A}\to \bar{B}$ for every such arrow $f$ of $\mathcal{C}$. Then $\mathcal{C}$ is naturally an $L$-structure by the identical interpretation of $L$ (i.e.~sending $\bar{A}$ to $A$ and $\bar{f}$ to $f$). More generally; each functor $F:\mathcal{C}\to \mathcal{D}$ creates an $L$-structure in $\mathcal{D}$.
\end{definition}

The following theorem (2.4.5. in \cite{makkai}) says, that from inside, $\mathcal{C}$ looks similar to $\mathbf{Set}$.

\begin{theorem}
Assume, that $\mathcal{C}$ has finite limits. Then the following diagrams in $\mathcal{C}$ have the stated properties, iff the sequents on their right side (have interpretation and) are valid (in $\mathcal{C}$, as a structure over its canonical language, with the identical interpretation for the signature).

\begin{tabularx}{\textwidth}{cXX}
1. & $A\xrightarrow{f} A$ is the identity on $A$ & $\Rightarrow f(a)\approx a$ \\ 
\hline 
2. & 
\begin{tikzpicture}
\node (3) at (0,0) {$A$};
\node[anchor=base,right=5mm of 3] (2a) {$C$};
\node[anchor=base,below=5mm of 3] (2b) {$B$};
\draw[->] (3)--(2a) node [midway,above] {$h$};
\draw[->] (3)--(2b) node [midway,left] {$f$};
\draw[->] (2b)--(2a) node [midway,below] {$g$};
\end{tikzpicture} is commutative & $\Rightarrow gf(a)\approx h(a)$\\ 
\hline 
3. & $A\xrightarrow{f} B$ is mono & $f(a)\approx f(a') \Rightarrow a\approx a'$ \\ 
\hline 
4. & $A\xrightarrow{f} B$ is surjective & $\Rightarrow \exists a:f(a)\approx b$ \\ 
\hline 
5. & $A$ is the terminal object & $\Rightarrow a\approx a'$ \\ & & $\Rightarrow \exists a: a\approx a$ \\ 
\hline 
6. & $A$ is the initial object & $a\approx a \Rightarrow $ \\ 
\hline 
7. & $A\xleftarrow{f} C\xrightarrow{g} B$ is a product diagram & \small{$f(c)\approx f(c') \wedge g(c)\approx g(c') \Rightarrow c\approx c'$} \\ & & $\Rightarrow \exists c (f(c)\approx a \wedge g(c)\approx b) $ \\ 
\hline 
8. & \begin{tikzpicture}
\node (1) at (0,0) {$E$};
\node[anchor=base,right=5mm of 1] (2) {$A$};
\node[anchor=base,right=5mm of 2] (3) {$B$};
\draw[right hook->] (1)--(2) node [midway,above] {$\epsilon$};
\draw[postaction={transform canvas={yshift=-1mm},draw}] [->] (2) -- (3) node [midway,above] {$f$} node [midway,below] {$g$};
\end{tikzpicture} is an equalizer & $f(a)\approx g(a) \Leftrightarrow \exists e: \epsilon (e)\approx a$ \\ 
\hline 
9. & \makecell[l]{$B\xhookrightarrow{g} X\xhookleftarrow{f_i} A_i$ ($i\in I$). \\ $B$ is the sup of $A_i$-s} & \small{$\bigvee _{i\in I} \exists a_i: f_i(a_i)\approx x \Leftrightarrow \exists b: g(b)\approx x$}\\ 
\hline 
10. & \makecell[l]{$B\xhookrightarrow{g} X\xhookleftarrow{f_i} A_i$ ($i\in I$). \\ $B$ is the inf of $A_i$-s} & \small{$\bigwedge _{i\in I} \exists a_i: f_i(a_i)\approx x \Leftrightarrow \exists b: g(b)\approx x$} \\ 
\end{tabularx}
\label{diagram}
\end{theorem}

\begin{definition}
Let $\mathcal{C}$ be a coherent category. Its (coherent) \emph{internal theory} $T_{\mathcal{C}}$ (or $Th(\mathcal{C})$) over the signature $L_{\mathcal{C}}$ consists of those sequents which refer to identities, commutative triangles, finite limits, surjective arrows and finite unions (as it is described above).
\end{definition}

\begin{theorem}
The categories $T_{\mathcal{C}}$-$mod(\mathcal{E})$ and $\mathbf{Coh}(\mathcal{C},\mathcal{E})$ are isomorphic (and the isomorphism is given by $M_0^*$ where $M_0:L_{\mathcal{C}}\to \mathcal{C}$ is the identical interpretation).
\end{theorem}

Now we will replace theories with categories. The idea behind the syntactic category is that a model can be seen as a collection of definable sets (evaluations of formulas) together with definable functions (whose graphs are definable). The syntactic category is the natural parametrisation of this category, i.e.~the one whose functorial images are precisely these collections.

The notion of derivability ($\vdash $) refers to a deduction system which is sound wrt.~every coherent category and which is complete wrt.~$\mathbf{Set}$-models, see \cite{makkai} for the details.

\begin{definition}
Let $T$ be a coherent theory. Its \emph{syntactic category} $\mathcal{C}_T$ is defined as follows:
\begin{itemize}
\item The objects are equivalence classes of coherent formulas (in context) over the given signature, where $\varphi (\vec{x})\sim \psi (\vec{y})$, iff $\psi (\vec{y})=\varphi (\vec{y}/ \vec{x})$. Note that $[\varphi (\vec{x})]$ and $[\varphi (\vec{x},\vec{y})]$ (with $y$ being an extra variable not present in $\varphi$) corresponds to different objects. This technicality is not essential, as $[\varphi (\vec{x},y)]$ turns out to be isomorphic with $[\varphi (\vec{x}) \wedge y\approx y]$, hence if we require all variables in the context $\vec{x}$ to appear freely in $\varphi $ we get an equivalent category.
\item An arrow $[\varphi (\vec{x})]\xrightarrow{[\theta (\vec{x},\vec{y})]} [\psi (\vec{y})]$ is an equivalence class of formulas, having the following properties:
\begin{itemize}
\item $\vec{x}$ and $\vec{y}$ are disjoint (this can always be assumed, as we can find such representatives of the objects),
\item $T\vdash \theta(\vec{x},\vec{y}) \Rightarrow \varphi (\vec{x}) \wedge \psi (\vec{y})$,
\item $T\vdash \varphi (\vec{x}) \Rightarrow \exists \vec{y} \theta(\vec{x},\vec{y})$,
\item $T\vdash \theta(\vec{x},\vec{y}) \wedge \theta(\vec{x},\vec{y'}) \Rightarrow \vec{y}=\vec{y'}$.
\end{itemize}
$\theta(\vec{x},\vec{y}) \sim \theta'(\vec{x'},\vec{y'})$, iff $T\vdash \theta(\vec{x},\vec{y}) \Leftrightarrow \theta'(\vec{x},\vec{y})$.
\end{itemize}

The identity arrow $1_{[\varphi (\vec{x}) ]}$ is given as $[\varphi (\vec{x})]\xrightarrow{[\varphi (\vec{x})\wedge \vec{x}\approx \vec{x'}]} [\varphi (\vec{x'})]$. The composition of $[\varphi (\vec{x})]\xrightarrow{[\theta (\vec{x},\vec{y})]} [\psi (\vec{y})] \xrightarrow{[\mu (\vec{y},\vec{z})]} [\chi (\vec{z})]$ is represented by $\exists \vec{y}(\theta (\vec{x},\vec{y}) \wedge \mu (\vec{y},\vec{z}))$.
\end{definition}

\begin{remark}
The required properties for $\theta $ are often referred as being "T-provably functional". This is because these are exactly the conditions which can guarantee, that the interpretation of $\theta $ in a model $M$ is not merely a subobject of $M(\varphi )\times M(\psi )\leq M(\vec{x})\times M(\vec{y})$, but the graph of an arrow from $M(\varphi )$ to $M(\psi )$.
\end{remark}

\begin{theorem}
Given a coherent theory $T$ over a signature $L$, its syntactic category $\mathcal{C}_T$ is a well-defined coherent category. The interpretation $M_0:L\to \mathcal{C}_T$ which maps a sort $X$ to the object $[x\approx x]$ (where $x$ is a variable of sort $X$), an arrow $f:X_1\times \dots \times X_n\to Y$ to $[x_1\approx x_1 \wedge \dots \wedge x_n\approx x_n]\xrightarrow{[f(x_1,\dots x_n)\approx y]}[y\approx y]$ and a relation symbol $R\subseteq X_1\times \dots \times X_n$ to the subobject $[R(x_1,\dots x_n)]\xhookrightarrow{[R(x_1,\dots x_n) \wedge x_1\approx x_1' \wedge \dots x_n\approx x_n']}[x_1'\approx x_1' \wedge \dots x_n'\approx x_n']$ is a model of $T$ with the property that $T\vdash \varphi \Rightarrow \psi $ iff $M_0\models \varphi \Rightarrow \psi$. 

The categories $T$-$mod(\mathcal{E})$ and $\mathbf{Coh}(\mathcal{C}_T,\mathcal{E})$ are equivalent. The equivalence is given by $M_0^*$ which takes a coherent functor $F:\mathcal{C}_T\to \mathcal{E}$ to the $L$-structure $F\circ M_0$ which is a model of $T$. The restriction of a natural transformation $\alpha :F\Rightarrow G$ to the image of $M_0$ yields a homomorphism of $L$-structures. Moreover $M_0^*$ is surjective on objects (not just essentially surjective).

If $\mathcal{C}$ is coherent then the interpretation of its canonical language $\mathcal{C}=L_{\mathcal{C}}\to \mathcal{C}_{Th(\mathcal{C})}$ is an equivalence.
\label{bigthm}
\end{theorem}

As a consequence we can prove that the forgetful 2-functor $\mathbf{U}:\mathbf{Coh}_{\sim }\to \mathbf{Cat}_{\sim }$ has a (2-categorical) left adjoint.

\begin{definition}
Let $\mathcal{A}$ be an ordinary category. We can see it as a signature whose sorts are the objects of $\mathcal{A}$ and whose unary function symbols are the morphisms. Then form the theory $Th(\mathcal{A})$ which consists of the sequents corresponding to the commutative triangles and identities in $\mathcal{A}$ (see Theorem \ref{diagram}). Let $\mathbf{F}(\mathcal{A})$ be its syntactic category and $\eta _{\mathcal{A}}:\mathcal{A}\to \mathbf{F}(\mathcal{A})$ be the interpretation of the signature (which is a functor when $\mathcal{A}$ is seen as a category).
\end{definition}

\begin{proposition}
$\mathbf{F}$ extends to a 2-functor $\mathbf{Cat}_{\sim}\to \mathbf{Coh}_{\sim }$ which is left adjoint to $\mathbf{U}$.
\end{proposition}

\begin{proof}
By the previous theorem $\eta _{\mathcal{A}}^*:\mathbf{Coh}(\mathbf{F}(\mathcal{A}), \mathcal{C})\to \mathbf{Cat}(\mathcal{A},\mathbf{U}(\mathcal{C}))\cong Th(\mathcal{A})$-$mod(\mathcal{C})$ is an equivalence of categories. As it reflects isomorphisms we can see it as an equivalence $\mathbf{Coh}_{\sim }(\mathbf{F}(\mathcal{A}), \mathcal{C})\to \mathbf{Cat}_{\sim}(\mathcal{A},\mathbf{U}(\mathcal{C}))$.
\end{proof}

We briefly list the constructions for finite limits, unions and image-facto\-ri\-sations in the syntactic category:

The terminal object is $[\top ]$. Binary products are given as
\[\begin{tikzcd}
	&& {[\chi (\vec{z})]} \\
	\\
	&& {[\varphi (\vec{x}) \wedge \psi (\vec{y})]} \\
	\\
	{[\varphi (\vec{x'})]} &&&& {[\psi (\vec{y'})]}
	\arrow["{[\varphi (\vec{x}) \wedge \psi (\vec{y}) \wedge \vec{x}\approx \vec{x'}]}"{description}, from=3-3, to=5-1]
	\arrow["{[\varphi (\vec{x}) \wedge \psi (\vec{y}) \wedge \vec{y}\approx \vec{y'}]}"{description}, from=3-3, to=5-5]
	\arrow["{[\mu (\vec{z},\vec{x'})]}"', curve={height=18pt}, from=1-3, to=5-1]
	\arrow["{[\nu (\vec{z},\vec{y'})]}", curve={height=-18pt}, from=1-3, to=5-5]
	\arrow[dashed, from=1-3, to=3-3]
\end{tikzcd}\]
where the dashed arrow is represented by $\mu (\vec{z},\vec{x}) \wedge \nu (\vec{z},\vec{y})$.

Equalizers are given as

\[\begin{tikzcd}
	{[\exists y (\mu (x,y)\wedge \nu (x,y))]} &&&& {[\varphi (x')]} && {[\psi (y)]} \\
	\\
	& {[\chi (z)]}
	\arrow["{[\exists y (\mu (x,y)\wedge \nu (x,y)) \wedge x\approx x']}", from=1-1, to=1-5]
	\arrow["{[\mu (x',y)]}", shift left=2, from=1-5, to=1-7]
	\arrow["{[\nu (x',y)]}"', shift right=2, from=1-5, to=1-7]
	\arrow["{[\tau (z,x')]}"', from=3-2, to=1-5]
	\arrow["{[\tau (z,x)]}", dashed, from=3-2, to=1-1]
\end{tikzcd}\]

It can be proved that every subobject of $[\psi (\vec{x})]$ can be represented by a monomorphism of the form $[\varphi  (\vec{x'})]\xrightarrow{[\varphi (\vec{x'}) \wedge \vec{x'}\approx \vec{x}]} [\psi (\vec{x})]$. Then unions are given as  $[\bigvee _i\varphi _i (\vec{x'})]\xrightarrow{[\bigvee _i\varphi _i(\vec{x'}) \wedge \vec{x'}\approx \vec{x}]} [\psi (\vec{x})]$.

Finally image-factorisations can be constructed as
\[\begin{tikzcd}
	{[\varphi (\vec{x})]} &&&& {[\psi (\vec{y})]} \\
	&& {[\exists \vec{x} \mu (\vec{x},\vec{y'})]}
	\arrow["{[\mu (\vec{x},\vec{y})]}", from=1-1, to=1-5]
	\arrow["{[\mu (\vec{x},\vec{y'})]}"'{pos=0.4}, from=1-1, to=2-3]
	\arrow["{[\exists \vec{x} \mu (\vec{x},\vec{y'}) \wedge \vec{y'}\approx \vec{y}]}"'{pos=0.6}, from=2-3, to=1-5]
\end{tikzcd}\]

We recall the idea of pretopos completion from \cite{makkai}.

\begin{definition}
A coherent functor $F:\mathcal{C}\to \mathcal{D}$ is said to be a \emph{weak equivalence} if $F^*:\mathbf{Coh}(\mathcal{D},\mathbf{Set})\to \mathbf{Coh}(\mathcal{C},\mathbf{Set})$ is an equivalence (i.e.~when $F$ is a Morita-equivalence). The class of weak equivalences is denoted by $W$.
\end{definition}

\begin{definition}
A coherent category $\mathcal{C}$ is a pretopos if it has
\begin{itemize}
    \item (finite) disjoint coproducts,
    \item quotients by equivalence relations, i.e.~given a subobject $j:R\hookrightarrow A\times A$ such that the axioms of reflectivity, symmetry and transitivity (see Definition \ref{jdef}) are valid in $\mathcal{C}$ (under the identical interpretation), the coequalizer of $\pi _1 j$ and $\pi _2j $ exists.
\end{itemize}
\end{definition}

\begin{theorem}
Given a coherent category $\mathcal{C}$ there exists a pretopos $\mathbf{R}\mathcal{C}$ and a coherent functor $\rho _{\mathcal{C}}:\mathcal{C}\to \mathbf{R}\mathcal{C}$ such that for any pretopos $\mathcal{S}$ the dashed arrow in
\[\begin{tikzcd}
	{\mathcal{C}} && {\mathbf{R}\mathcal{C}} \\
	\\
	&& {\mathcal{S}}
	\arrow["{\rho _{\mathcal{C}}}", from=1-1, to=1-3]
	\arrow[""{name=0, anchor=center, inner sep=0}, "F"', from=1-1, to=3-3]
	\arrow["\mathbf{R}F", dashed, from=1-3, to=3-3]
	\arrow["\alpha", shorten <=6pt, shorten >=6pt, Rightarrow, from=1-3, to=0]
\end{tikzcd}\]
exists and it is essentially unique: given $F'$ and $\beta :F'\rho _{\mathcal{C}}\Rightarrow F$ there is a unique natural isomorphism $\eta :\mathbf{R}F\Rightarrow F'$ such that the pasting of $\beta $ and $\eta $ is $\alpha $. Moreover $\rho _{\mathcal{C}}$ is a weak equivalence.
\end{theorem}

(This appears as Theorem 8.4.1.~of \cite{makkai}. The fact that $\rho _{\mathcal{C}}$ is a weak equivalence can be found as Proposition 9., Lecture 13.~in \cite{lurie}.)

\begin{corollary}
The full (2,1)-subcategory of pretoposes is reflective in the (2,1)-categorical sense.
\end{corollary}

We recall Makkai's conceptual completeness theorem:

\begin{theorem}
Let $\mathcal{C}$ be a pretopos. A coherent functor $F:\mathcal{C}\to \mathcal{D}$ is an equivalence iff it is a weak equivalence.
\label{conceptual}
\end{theorem}

This can be reformulated as:

\begin{theorem}
A coherent functor $F:\mathcal{C}\to \mathcal{D}$ is a weak equivalence iff $\mathbf{R}F$ is an equivalence.
\end{theorem}

\begin{proof}
Since $W$ clearly has the 2-for-3 property $F$ is a weak equivalence iff $\mathbf{R}F$ is a weak equivalence. By the conceptual completeness theorem the latter is equivalent to $\mathbf{R}F$ being an equivalence.
\end{proof}

As $\mathbf{R}$ is a left adjoint it preserves 2-colimits hence we get:

\begin{corollary}
$W$ is closed under transfinite composition, pushouts and retracts.
\label{wclosed}
\end{corollary}

\section{The model structure}

We will apply the previous results to provide a (2,1)-model structure for $\mathbf{Coh}_{\sim }$, using the (2,1)-categorical small object argument. However, there is another possibility. Theorem 3.3.~of \cite{reflective} proves in the 1-categorical context that given a finitely well-complete category $\mathbf{C}$ (i.e.~finitely complete with all intersections) and a reflective subcategory $\mathbf{A}$, then taking $W$ to be the class of maps inverted by the reflector, the pair $(W,W$-$inj$ $)$ is a factorisation system and any such (so-called reflective) factorisation system results a model structure by taking $W$ to be the class of weak equivalences, $W$-$inj$ to be the class of fibrations and all maps to be cofibrations. Following Remark 3.8.~and Example 2.4.~of \cite{cellular} $W$ is proved to be cofibrantly generated in this case.
The advantage of our approach is that it gives an explicit description of the generating trivial cofibrations and the (2,1)-categorical small object argument should be of independent interest. Using the results of \cite{makkai} we will also prove the resulting model structure to be right proper. 

\begin{definition}
Take the signature $L_R $ with one sort $A$ and a binary relation symbol $R\subseteq A\times A$. Let $\mathcal{C}_R$ be the syntactic category of the theory

\begin{tabularx}{\textwidth}{XX}
$\top \Rightarrow R(a,a)$ & ($R$ is reflective) \\ 
$R(a,a')\Rightarrow R(a',a)$ & ($R$ is symmetric) \\ 
$R(a,a')\wedge R(a',a'') \Rightarrow R(a,a'')$ & ($R$ is transitive) \\ 
\end{tabularx}
We set $M_0:L_R \to \mathcal{C}_R$ to denote the canonical interpretation of the signature.

Then add a sort $B$ and a function symbol $p:A\to B$ to the signature (to form $L_{A/R}$) and add the axioms

\begin{tabularx}{\textwidth}{XX}
$\top \Rightarrow \exists a (b\approx p(a))$ & ($p$ is surjective) \\ 
$R(a,a')\Leftrightarrow p(a)\approx p(a')$ & ($p$ identifies precisely the $R$-equivalent elements) \\ 
\end{tabularx}
to the theory. Now form its syntactic category $\mathcal{C}_{A/R}$ with the canonical interpretation $M_0':L_{A/R}\to \mathcal{C}_{A/R}$ of the extended signature. Let $I:L_R\to L_{A/R}$ be the inclusion of the signatures. As $M_0'\circ I$ is a model of the defining theory of $\mathcal{C}_R$ by the surjectivity of $M_0^*$ on objects (Theorem \ref{bigthm}) we have a commutative square:
\[\begin{tikzcd}
	{L_R} && {L_{A/R}} \\
	& = \\
	{\mathcal{C}_R} && {\mathcal{C}_{A/R}}
	\arrow["I", from=1-1, to=1-3]
	\arrow["{M_0}"', from=1-1, to=3-1]
	\arrow["{M_0'}", from=1-3, to=3-3]
	\arrow["{M_{A/R}}"', dashed, from=3-1, to=3-3]
\end{tikzcd}\]
(and $M_{A/R}$ is unique up to isomorphism with the property that the above square commutes up to isomorphism, as $M_0^*$ is fully faithful).

Finally take the signature with sorts $A,B,S$ and unary function symbols $i_1:S\to A$ and $i_2:S \to B$. Let $\mathcal{C}_{A\cap B}$ be the syntactic category of the theory
\begin{tabularx}{\textwidth}{XX}
$i_{1,2}(s)\approx i_{1,2}(s') \Rightarrow s\approx s'$ & ($i_1$ and $i_2$ are monic)
\end{tabularx}
Now extend the signature with a new sort $X$ and unary function symbols $j_1:A\to X$, $j_2:B\to X$. Let $\mathcal{C}_{cov}$ be the syntactic category of the theory extending the previous sequents with
\begin{tabularx}{\textwidth}{XX}
$j_{1}(a)\approx j_{1}(a')\Rightarrow a\approx a'$ & ($j_1$ is monic) \\ 
$j_{2}(b)\approx j_{2}(b')\Rightarrow b\approx b'$ & ($j_2$ is monic) \\ 
 $ \top \Rightarrow \exists a (x\approx j_1(a)) \vee \exists b (x\approx j_2(b))$  & ($j_1$ and $j_2$ jointly cover $X$) \\ 
$j_1(a)\approx j_2(b) \Leftrightarrow \exists s: a\approx i_1(s) \wedge b\approx i_2(s)$ & \makecell{($j_1$ and $j_2$ identify precisely \\ the elements of $S$)} \\ 
\end{tabularx}
Let $M_{cov}:\mathcal{C}_{A\cap B}\to \mathcal{C}_{cov}$ be the evident interpretation (induced as before).

We set $J=\{M_{A/R}, M_{cov} \}$.
\label{jdef}
\end{definition}

\begin{definition}
The elements of $J$-$cof$ are called \emph{trivial cofibrations}, and the elements of $J$-$inj$ are called \emph{fibrations}. A coherent category $\mathcal{C}$ is called \emph{fibrant} if $\mathcal{C}\to *$ is a fibration.
\end{definition}

\begin{theorem}
Any coherent functor $M:\mathcal{C}\to \mathcal{D}$ factors as $M\cong GF$ where $F$ is a trivial cofibration (moreover $F\in J$-cell) and $G$ is a fibration.  
\end{theorem}



\begin{proposition}
Trivial cofibrations are closed under 2-pushouts. Fibrations are closed under 2-pullbacks.
\end{proposition}

\begin{proof}
Follows from Proposition \ref{lifting}.
\end{proof}

\begin{proposition}
Let $F:\mathcal{C}\to \mathcal{D}$ be a trivial cofibration. Then there is a map $G:\mathcal{C}\to \mathcal{D}'$ in $J$-cell such that $F$ is the retract of $G$.
\label{retract}
\end{proposition}

\begin{proof}
We can write $F\cong HG$ where $G$ is the strict transfinite composition of pushouts from $J$ and $H$ is a fibration. Then there is a lifting in 
\[\begin{tikzcd}
	{\mathcal{C}} && {\mathcal{D}'} \\
	\\
	{\mathcal{D}} && {\mathcal{D}}
	\arrow["F"', from=1-1, to=3-1]
	\arrow[""{name=0, anchor=center, inner sep=0}, "G", from=1-1, to=1-3]
	\arrow["H", from=1-3, to=3-3]
	\arrow[""{name=1, anchor=center, inner sep=0}, Rightarrow, no head, from=3-1, to=3-3]
	\arrow[dashed, from=3-1, to=1-3]
	\arrow[shorten <=18pt, shorten >=18pt, Rightarrow, from=0, to=3-1]
	\arrow[shorten <=22pt, shorten >=13pt, Rightarrow, from=1-3, to=1]
\end{tikzcd}\]
which exhibits $F$ as the retract of $G$.
\end{proof}

\begin{theorem}
$J$-$cof$$=W$
\end{theorem}

\begin{proof}
$\subseteq $ By Corollary \ref{wclosed} it is enough to prove $J\subseteq W$. Let $N:\mathcal{C}_{R}\to \mathbf{Set}$ be a coherent functor. Then $NM_0$ gives a set $N(A)$ with an equivalence relation $N(R)$ on it. If we want to extend this to an $L_{A/R}$-structure which is a model of the additional axioms we are forced to interpret $p$ as the coequalizer of $\restr{\pi _1}{N(R)}$ and $\restr{\pi _2}{N(R)}$. A homomorphism of $L_R$-structures (which are models of the axioms on $R$) is a commutative diagram
\[\begin{tikzcd}
	{N(R)} && {N(A)\times N(A)} \\
	& {=} \\
	{N'(R)} && {N'(A)\times N'(A)}
	\arrow[hook, from=1-1, to=1-3]
	\arrow[hook, from=3-1, to=3-3]
	\arrow["{h\times h}"{description}, from=1-3, to=3-3]
	\arrow[dashed, from=1-1, to=3-1]
\end{tikzcd}\]
which induces a unique map between the coequalizers. The case of $M_{cov}$ is analogous.

$\supseteq $ Let $H:\mathcal{C}\to \mathcal{D}$ be a weak equivalence. We know from \cite{makkai} that $H$ is fully faithful, conservative, full wrt.~subobjects and each object $x\in \mathcal{D}$ is finitely covered by $\mathcal{C}$ via $H$, i.e.~there are objects $y_1,\dots y_n$ in $\mathcal{C}$, subobjects $b_k\hookrightarrow H(y_k)$ and maps $p_k:b_k\to x$ such that $x=\bigcup \exists _{p_k}b_k$. Since $H$ is full wrt.~subobjects we can take $b_k=H(y_k)$.

Let $(d_i)_{i<\lambda }$ be a well-ordering of objects of $\mathcal{D}$. By transfinite recursion we will give factorisations $H\cong \mathcal{C}\xrightarrow{F_i}\mathcal{D}_i\xrightarrow{G_i}\mathcal{D}$ such that $F_i\in J$-cell and for each $j<i$ $d_j$ is contained in the essential image of $G_i$. This is sufficient as $J-cell \subseteq W $ so by the 2-out-3 property $G_{\lambda }$ is an essentially surjective weak equivalence, i.e.~an equivalence.

We take $\mathcal{D}_0=\mathcal{C}, F_0=1_{\mathcal{C}}, G_0=h$. When $i$ is a limit ordinal let $F_i$ be the transfinite composition of $(F_j)_{j<i}$ and $G_i$ be the induced map. Now assume that $\mathcal{C}\xrightarrow{F_i}\mathcal{D}_i\xrightarrow{G_i}\mathcal{D}$ is given. We can find $y_1,\dots y_n$ in $\mathcal{D}_i$ such that in 
\[\begin{tikzcd}
	{G_i(y_1)} \\
	\dots && {d_i} \\
	{G_i(y_n)}
	\arrow["{p_1}"{description}, from=1-1, to=2-3]
	\arrow["{p_n}"{description}, from=3-1, to=2-3]
\end{tikzcd}\]
$d_i=\bigcup _k \exists _{p_k}G_i(y_k)$. First we glue the quotient maps $q_i:G_i(y_k)\twoheadrightarrow \exists _{p_k}G_i(y_k)$ to $\mathcal{D}_i$, i.e.~take the 2-pushout and the induced map in
\[\begin{tikzcd}
	{\mathcal{C}_R} & {\mathcal{D}_i} \\
	{\mathcal{C}_{A/R}} & {\mathcal{D}_i^1} \\
	&&& {\mathcal{D}}
	\arrow["{M_{A/R}}"', from=1-1, to=2-1]
	\arrow[from=1-1, to=1-2]
	\arrow[from=1-2, to=2-2]
	\arrow[from=2-1, to=2-2]
	\arrow[curve={height=6pt}, from=2-1, to=3-4]
	\arrow["{G_i}", curve={height=-6pt}, from=1-2, to=3-4]
	\arrow["{G_i^1}"{description}, dashed, from=2-2, to=3-4]
\end{tikzcd}\]
where $R\hookrightarrow A$ is mapped to the subobject $R'\hookrightarrow y_1$ whose $G_i$-image is the kernel pair of $q_1$ (and which is an equivalence relation as $G_i$ is bijective on the subobject lattices) and $A\xrightarrow{p}B$ is mapped to $q_1$. The dashed arrow is a weak equivalence and $q_1$ lies in its essential image. Iterating it $n$ times yields a factorisation $H\cong \mathcal{C}\xrightarrow{F_i^n}\mathcal{D}_i^n\xrightarrow{G_i^n} \mathcal{D}$ with $F_i^n\in J$-cell and with $(d_j)_{j<i}$ and $\exists _{p_1}G_i(y_1),\dots \exists _{p_n}G_i(y_n)$  all lying in the essential image of $G_i^n $. 

Using that $G_i^n$ is a weak equivalence and hence it is full wrt.~subobjects and it is fully faithful we have a covering of $d_i$
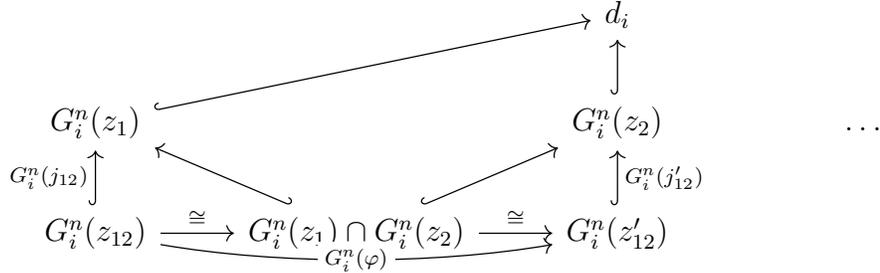
\[\begin{tikzcd}
	&& {d_i} \\
	{G_i^n(z_1)} && {G_i^n(z_2)} && \dots \\
	{G_i^n(z_{12})} & {G_i^n(z_1)\cap G_i^n(z_2)} & {G_i^n(z_{12}')}
	\arrow["{G_i^n(j_{12})}", hook, from=3-1, to=2-1]
	\arrow[hook, from=2-1, to=1-3]
	\arrow[hook, from=2-3, to=1-3]
	\arrow[hook', from=3-2, to=2-1]
	\arrow[hook, from=3-2, to=2-3]
	\arrow["\cong", from=3-1, to=3-2]
	\arrow["\cong", from=3-2, to=3-3]
	\arrow["{G_i^n(j_{12}')}"', hook, from=3-3, to=2-3]
	\arrow["{G_i^n(\varphi)}"{description}, curve={height=12pt}, from=3-1, to=3-3]
\end{tikzcd}\]
Now we just glue the unions of the elements of the cover to $\mathcal{D}_i^n$. Take the pushout and the induced map in 
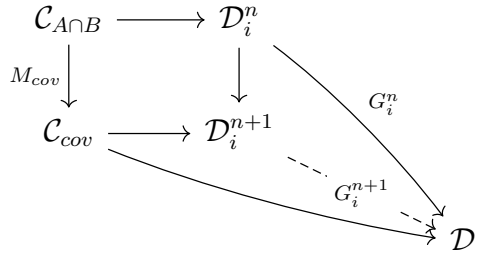
\[\begin{tikzcd}
	{\mathcal{C}_{A\cap B}} & {\mathcal{D}_i^n} \\
	{\mathcal{C}_{cov}} & {\mathcal{D}_i^{n+1}} \\
	&&& {\mathcal{D}}
	\arrow["{M_{cov}}"', from=1-1, to=2-1]
	\arrow[from=1-1, to=1-2]
	\arrow[from=1-2, to=2-2]
	\arrow[from=2-1, to=2-2]
	\arrow[curve={height=6pt}, from=2-1, to=3-4]
	\arrow["{G_i^n}", curve={height=-6pt}, from=1-2, to=3-4]
	\arrow["{G_i^{n+1}}"{description}, dashed, from=2-2, to=3-4]
\end{tikzcd}\]
where $A\leftarrow S\to B$ is sent to $z_1\xleftarrow{j_{12}}z_{12}\xrightarrow{j'_{12}\varphi }z_2$ and $A\to X\leftarrow B$ is mapped to $G_i^n(z_1)\to d_i \leftarrow G_i^n(z_2)$. It follows that $d_i$ is covered by $n-1$ elements from the essential image of $G_i^{n+1}$. We can take $\mathcal{D}_{i+1}=\mathcal{D}_i^{2n-1}$, $F_{i+1}=F_i^{2n-1}$ and $G_{i+1}=G_i^{2n-1}$.
\end{proof}

\begin{corollary}
Every map $F\in W\cap J$-$inj$ is an equivalence.
\end{corollary}

\begin{proof}
As $W=J$-$cof$ we have a lift in
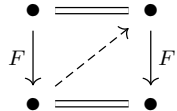
\[\begin{tikzcd}
	\bullet & \bullet \\
	\bullet & \bullet
	\arrow["F"', from=1-1, to=2-1]
	\arrow[Rightarrow, no head, from=1-1, to=1-2]
	\arrow["F", from=1-2, to=2-2]
	\arrow[Rightarrow, no head, from=2-1, to=2-2]
	\arrow[dashed, from=2-1, to=1-2]
\end{tikzcd}\]
\end{proof}

Following the proof of Proposition 2.3.~in \cite{torsion} we get:

\begin{proposition}
The lift in the square
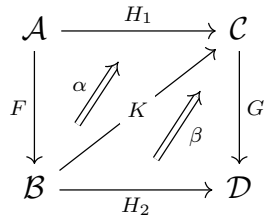
\[\begin{tikzcd}
	{\mathcal{A}} && {\mathcal{C}} \\
	\\
	{\mathcal{B}} && {\mathcal{D}}
	\arrow["F"', from=1-1, to=3-1]
	\arrow[""{name=0, anchor=center, inner sep=0}, "{H_1}", from=1-1, to=1-3]
	\arrow["G", from=1-3, to=3-3]
	\arrow[""{name=1, anchor=center, inner sep=0}, "{H_2}"', from=3-1, to=3-3]
	\arrow["K"{description}, from=3-1, to=1-3]
	\arrow["\alpha", shorten <=18pt, shorten >=13pt, Rightarrow, from=3-1, to=0]
	\arrow["\beta"', shorten <=13pt, shorten >=18pt, Rightarrow, from=1, to=1-3]
\end{tikzcd}\]
with $F\in W$, $G\in J$-$inj$ is essentially unique: given $K':\mathcal{B}\to \mathcal{C}$ and $\alpha '$, $\beta '$ with $\beta F \circ G\alpha =\beta 'F \circ G \alpha '$ there is a unique natural isomorphism $\gamma :K\Rightarrow K'$ for which $\alpha =\gamma F \circ \alpha '$ and $\beta '=G\gamma \circ \beta $.
\end{proposition}

\begin{proof}
The 2-cells
\[\begin{tikzcd}
	{\mathcal{A}} & {\mathcal{B}} &&&& {\mathcal{A}} & {\mathcal{B}} \\
	{\mathcal{B}} & {\mathcal{B}'} && {\mathcal{B}} & {\mathcal{D}} & {\mathcal{B}} & {\mathcal{B}'} && {\mathcal{C}} & {\mathcal{D}}
	\arrow["F"', from=1-1, to=2-1]
	\arrow[""{name=0, anchor=center, inner sep=0}, "F", from=1-1, to=1-2]
	\arrow["{I_1}", from=1-2, to=2-2]
	\arrow[""{name=1, anchor=center, inner sep=0}, "{I_2}", from=2-1, to=2-2]
	\arrow[""{name=2, anchor=center, inner sep=0}, "{1_{\mathcal{B}}}"'{pos=0.8}, curve={height=24pt}, from=2-1, to=2-4]
	\arrow[""{name=3, anchor=center, inner sep=0}, "{1_{\mathcal{B}}}", curve={height=-6pt}, from=1-2, to=2-4]
	\arrow[""{name=4, anchor=center, inner sep=0}, "{F'}"{description}, dashed, from=2-2, to=2-4]
	\arrow["{H_2}", from=2-4, to=2-5]
	\arrow[""{name=5, anchor=center, inner sep=0}, "F", from=1-6, to=1-7]
	\arrow["F"', from=1-6, to=2-6]
	\arrow["{I_1}", from=1-7, to=2-7]
	\arrow[""{name=6, anchor=center, inner sep=0}, "{I_2}", from=2-6, to=2-7]
	\arrow[""{name=7, anchor=center, inner sep=0}, "{K'}"{description}, curve={height=18pt}, from=2-6, to=2-9]
	\arrow[""{name=8, anchor=center, inner sep=0}, "K"{description}, curve={height=-6pt}, from=1-7, to=2-9]
	\arrow[""{name=9, anchor=center, inner sep=0}, "R"{description}, dashed, from=2-7, to=2-9]
	\arrow["G", from=2-9, to=2-10]
	\arrow[""{name=10, anchor=center, inner sep=0}, "{H_2}", curve={height=-12pt}, from=1-7, to=2-10]
	\arrow[""{name=11, anchor=center, inner sep=0}, "{H_2}"', curve={height=30pt}, from=2-6, to=2-10]
	\arrow["push"{description}, Rightarrow, draw=none, from=0, to=1]
	\arrow["push"{description}, Rightarrow, draw=none, from=5, to=6]
	\arrow["{\varepsilon_1}"{description}, Rightarrow, draw=none, from=3, to=2-2]
	\arrow["{\mu_1}"{description}, Rightarrow, draw=none, from=8, to=2-7]
	\arrow["{\mu_2}"{description, pos=0.4}, Rightarrow, draw=none, from=9, to=7]
	\arrow["{\varepsilon _2}"{description}, Rightarrow, draw=none, from=4, to=2]
	\arrow["\beta"{description}, Rightarrow, draw=none, from=10, to=2-9]
	\arrow["{\beta '}"{description}, Rightarrow, draw=none, from=2-9, to=11]
\end{tikzcd}\]
are identical, i.e.~we have $H_2F\xRightarrow{\beta 'F} GK'F \xRightarrow{G \alpha ^{-1} \circ G \alpha '}GKF \xRightarrow{\beta ^{-1}F} H_2F$ equals $1_{H_2F}$. By the universal property of the pushout there is a unique 2-cell
\[\begin{tikzcd}
	{\mathcal{B}'} && {\mathcal{C}} \\
	\\
	{\mathcal{B}} && {\mathcal{D}}
	\arrow["{F'}"', from=1-1, to=3-1]
	\arrow["R", from=1-1, to=1-3]
	\arrow["G", from=1-3, to=3-3]
	\arrow["{H_2}"', from=3-1, to=3-3]
	\arrow["\eta"{description}, shorten <=11pt, shorten >=11pt, Rightarrow, from=3-1, to=1-3]
\end{tikzcd}\]
with $H_2\xRightarrow{\beta } GK \xRightarrow{ G\mu _1}GRI_1 \xRightarrow{\eta ^{-1}I_1}H_2F'I_1$ equal to $H_2\varepsilon _1$ and with $H_2\xRightarrow{\beta '}GK \xRightarrow{G\mu _2} GRI_2\xRightarrow{ \eta ^{-1}I_2}H_2F'I_2 $ equal to $H_2\varepsilon _2$.

As weak equivalences are closed under 2-pushouts and by the 2-for-3 property the map $F'$ is a weak equivalence, hence we have a lift
\[\begin{tikzcd}
	{\mathcal{B}'} && {\mathcal{C}} \\
	\\
	{\mathcal{B}} && {\mathcal{D}}
	\arrow[""{name=0, anchor=center, inner sep=0}, "R", from=1-1, to=1-3]
	\arrow["G", from=1-3, to=3-3]
	\arrow[""{name=1, anchor=center, inner sep=0}, "{H_2}"', from=3-1, to=3-3]
	\arrow["L"{description}, from=3-1, to=1-3]
	\arrow["{F'}"', from=1-1, to=3-1]
	\arrow["{\nu _1}", shorten <=18pt, shorten >=13pt, Rightarrow, from=3-1, to=0]
	\arrow["{\nu _2}"', shorten <=18pt, shorten >=13pt, Rightarrow, from=1, to=1-3]
\end{tikzcd}\]
with $G\nu _1 \circ \nu _2 F'= \eta $.

Then we have isomorphisms:
\[
\gamma : K\xRightarrow{\mu _1} RI_1 \xRightarrow{\nu _1^{-1}I_1} LF'I_1 \xRightarrow{L\varepsilon _1^{-1}} L \xRightarrow{L\varepsilon _2} LF'I_2 \xRightarrow{\nu _1 I_2} RI_2 \xRightarrow{\mu _2^{-1}} K'
\]
for which $\alpha =\gamma F \circ \alpha '$ and $\beta '=G\gamma \circ \beta $ is easily checked.

Given $\gamma '$ with these properties we get $\delta \neq \delta ':K\Rightarrow L$ then $\eta \neq G\eta _1 \circ G \delta ^{-1}\delta ' F' \circ \nu _2 F'$ are both compatible 2-cells which contradicts the universal property of the 2-pushout.
\end{proof}

\begin{proposition}
Assume $GF\cong H$. Then if $G,H$ are fibrations, the map $F$ is also a fibration.
\label{cancellation}
\end{proposition}

\begin{proof}
We have a lift $L$ in
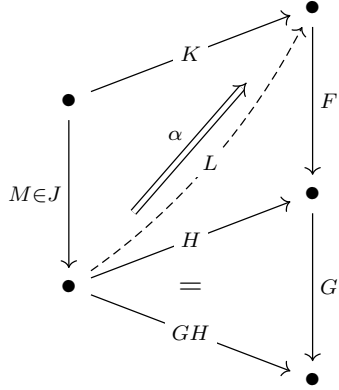
\[\begin{tikzcd}
	&& \bullet \\
	\bullet \\
	&& \bullet \\
	\bullet & {=} \\
	&& \bullet
	\arrow["{M \in J}"', from=2-1, to=4-1]
	\arrow["F", from=1-3, to=3-3]
	\arrow["G", from=3-3, to=5-3]
	\arrow["K"{description}, from=2-1, to=1-3]
	\arrow["H"{description}, from=4-1, to=3-3]
	\arrow["GH"{description}, from=4-1, to=5-3]
	\arrow["\alpha", shorten <=28pt, shorten >=28pt, Rightarrow, from=4-1, to=1-3]
	\arrow["L"{description}, curve={height=12pt}, dashed, from=4-1, to=1-3]
\end{tikzcd}\]
with isomorphisms $\nu _1:LM\Rightarrow K$, $\nu _2:GH\Rightarrow GFL$ such that $\nu _2 M \circ GF\nu _1 =G\alpha $. When seeing it as a square whose right edge is $G$ we have two liftings $FL$ with 2-cells $\nu _1$, $\nu _2$ and $H$ with 2-cells $\alpha $, $1_{GH}$. By the previous proposition we get a unique isomorphism $\gamma :H\Rightarrow FL$ with $\gamma M \circ F\nu _1=\alpha $ and $G\gamma =\nu _2$ so $(L,\nu _1, \gamma )$ is a splitting of the 2-cell $\alpha $.
\end{proof}

\begin{definition}
A \emph{path object} for a coherent category $\mathcal{C}$ is the factorisation of $\Delta :\mathcal{C}\to \mathcal{C}\times \mathcal{C}$ as $\mathcal{C}\xrightarrow{I_0} \mathcal{C}' \xrightarrow{\langle P_0,P_1 \rangle } \mathcal{C}\times \mathcal{C}$ where $I_0$ is a weak equivalence and $\langle P_0,P_1 \rangle $ is a fibration.

Let $\mathcal{C},\mathcal{D}$ be fibrant. A \emph{homotopy} between two coherent functors $F,G:\mathcal{C}\to \mathcal{D}$ is a coherent functor $K:\mathcal{C}\to \mathcal{D}'$ with $P_0K\cong F$ and $P_1K\cong G$. We say that $F$ and $G$ are homotopic (and write $F\simeq G$) if there is a homotopy between them with some path object.
\end{definition}

\begin{theorem}
There is a (2,1)-model structure on $\mathbf{Coh}_{\sim }$ with $W$ being the class of weak equivalences, $Fib=J$-$inj$ and with all maps being cofibrations. A coherent category is fibrant iff it is a pretopos. Two maps $F,G:\mathcal{C}\to \mathcal{D}$ between pretoposes are homotopic iff they are naturally isomorphic.
\label{mainthm}
\end{theorem}

\begin{proof}
$W$ satisfies the 2-out-3 and the retract axioms, the elements of $J$-$inj$ have the right lifting property wrt.~$W\cap Cof =W=J$-$cof$, equivalences (i.e.~elements of $W\cap Fib$) have the right lifting property wrt.~anything and the existence of the nontrivial factorisation system has already been noticed and it follows from the (2,1)-categorical small object argument.

$\mathcal{C}\to *$ is a fibration iff it reflects quotients by equivalence relations and pushouts of monomorphisms along monomorphisms which are also pullbacks, equivalently if these constructions exist in $\mathcal{C}$. Taking monomorphisms with initial domain this implies the existence of disjoint coproducts (which are pullback-stable by the stability of unions). To see the converse we will show that any map out of a pretopos is a fibration. I.e.~let
\[\begin{tikzcd}
	{\mathcal{C}} && {\mathcal{P}} \\
	& {\mathcal{Q}} \\
	{\mathcal{D}} && {\mathcal{E}}
	\arrow["F", from=1-1, to=3-1]
	\arrow[from=1-1, to=1-3]
	\arrow["{F'}"', from=1-3, to=2-2]
	\arrow[from=3-1, to=2-2]
	\arrow["G", from=1-3, to=3-3]
	\arrow[from=3-1, to=3-3]
	\arrow[dashed, from=2-2, to=3-3]
\end{tikzcd}\]
be a (homotopy) commutative square where $\mathcal{P}$ is a pretopos and $F$ is a weak equivalence, with $\mathcal{Q}$ being the 2-pushout. Then $F'$ is a weak equivalence, hence it is an equivalence as $\mathcal{P}$ is a pretopos (using Makkai's conceptual completeness: Theorem 7.1.8.~in \cite{makkai}). Its quasi-inverse gives a splitting of the original 2-cell.

It follows that when $\mathcal{D}$ is a pretopos the diagonal $\Delta :\mathcal{D}\to \mathcal{D}\times \mathcal{D}$ is a fibration, hence $(1_{\mathcal{D}},\Delta )$ is a path object for $\mathcal{D}$.

\end{proof}

\begin{remark}
The (2,1)-categorical analogue of Whitehead's theorem says that a map between objects that are both fibrant and cofibrant, is a weak equivalence iff it is a homotopy equivalence. Therefore the given model structure exhibits conceptual completeness (Theorem \ref{conceptual}, for coherent functors where the codomain is also a pretopos) as an instance of Whitehead's theorem. 

Fibrant replacement gives pretopos completion.
\end{remark}

\begin{remark}
It is not surprising that there is a set $J$ of generating trivial cofibrations for which pretoposes are precisely the fibrant objects. Indeed, the full subcategory of pretoposes is an accessibly embedded accessible subcategory of $\mathbf{Coh}_{\sim }$ which is closed under finite products, and $\mathbf{Coh}_{\sim }$ is locally presentable in the 2-dimensional sense. Then by a 2-dimensional version of Theorem 4.8.~of \cite{rosicky} pretoposes form a small-injectivity class.
\end{remark}

\begin{proposition}
Fully faithful functors are closed under 2-pullbacks in $\mathbf{Coh}_{\sim }$.
\label{fullf}
\end{proposition}

\begin{proof}
By \cite{codescent} bijective-on-objects and fully faithful functors form a 2-categorical orthogonal factorisation system on $\mathbf{Cat}_{\sim }$, hence the right class is stable under 2-pullbacks. By \cite{flexible} the forgetful functor $\mathbf{U}:\mathbf{Coh}_{\sim }\to \mathbf{Cat}_{\sim }$ preserves and reflects 2-limits.
\end{proof}

\begin{proposition}
$F:\mathcal{C}\to \mathcal{D}$ is a fibration iff 
\[\begin{tikzcd}
	{\mathcal{C}} && {\mathbf{R}\mathcal{C}} \\
	\\
	{\mathcal{D}} && {\mathbf{R}\mathcal{D}}
	\arrow["F"', from=1-1, to=3-1]
	\arrow["{\rho _{\mathcal{C}}}", from=1-1, to=1-3]
	\arrow["\mathbf{R}F", from=1-3, to=3-3]
	\arrow["{\rho _{\mathcal{D}}}"', from=3-1, to=3-3]
\end{tikzcd}\]
is a 2-pullback.
\end{proposition}

\begin{proof}
Since any map out of a pretopos is a fibration the "if" direction is immediate. For the converse assume that $F$ is a fibration and take the pullback
\[\begin{tikzcd}
	{\mathcal{C}} \\
	& S && {\mathbf{R}\mathcal{C}} \\
	\\
	& {\mathcal{D}} && {\mathbf{R}\mathcal{D}}
	\arrow["F"', curve={height=12pt}, from=1-1, to=4-2]
	\arrow["{\rho _{\mathcal{C}}}", curve={height=-12pt}, from=1-1, to=2-4]
	\arrow["\mathbf{R}F", from=2-4, to=4-4]
	\arrow["{\rho _{\mathcal{D}}}"', from=4-2, to=4-4]
	\arrow["{F'}"', from=2-2, to=4-2]
	\arrow["U"', from=2-2, to=2-4]
	\arrow["V"', dashed, from=1-1, to=2-2]
\end{tikzcd}\]

We should prove that $V$ is an equivalence. As both $F$ and $F'$ are fibrations $V$ is a fibration by Proposition \ref{cancellation}. Hence it suffices to prove that $V$ is a weak equivalence.

By \cite{makkai} any weak equivalence is fully faithful. By Proposition \ref{fullf} $U$ is fully faithful and since fully faithful functors form the right class of a factorisation system on $\mathbf{Cat}_1$ we get that so is $V$. Then it suffices to prove essential surjectivity.

Take $x\in S$. Referring to \cite{makkai} again $Ux$ is finitely covered by $\mathcal{C}$ via $\rho _{\mathcal{C}}$, i.e.~there are objects $c_1,\dots c_n\in \mathcal{C}$ and maps $f_i:\rho _{\mathcal{C}}(c_i)\to Ux$ such that $\bigvee _i \exists _{f_i}\rho _{\mathcal{C}}c_i =Ux$. Using that $U$ is fully faithful we get maps $g_i:Vc_i\to x$ with the same property.

As $\rho _{\mathcal{C}}$ is full wrt.~subobjects the composite $Sub(c)\hookrightarrow Sub(Vc)\hookrightarrow Sub(UVc)$ is surjective, therefore both maps are bijections, in particular $V$ is full wrt.~subobjects. Taking $k_i\hookrightarrow Vc_i\times Vc_i $ to be the kernel pair of $g_i$ we have that it is coming from an equivalence relation on $c_i\times c_i$. As $V$ has the right lifting property against $M_{A/R}$ we can choose $g_i$'s to be monomorphisms. Then orthogonality against $M_{cov}$ completes the proof.

\end{proof}

\begin{corollary}
The model structure of Theorem \ref{mainthm} is right proper.
\end{corollary}

\begin{proof}
By the pasting law of pullbacks and by the previous proposition $\rho _{\mathcal{B}} U'$ is a weak equivalence.

\[\begin{tikzcd}
	{\mathcal{S}} && {\mathcal{B}} && {\mathbf{R}\mathcal{B}} & {\mathbf{R}\mathcal{S}} \\
	& pb && pb \\
	{\mathcal{A}} && {\mathcal{C}} && {\mathbf{R}\mathcal{C}} & {\mathbf{R}\mathcal{A}}
	\arrow["U"', from=3-1, to=3-3]
	\arrow["{F'}"', from=1-1, to=3-1]
	\arrow["{U'}", from=1-1, to=1-3]
	\arrow["F"', from=1-3, to=3-3]
	\arrow["{\rho _{\mathcal{B}}}", from=1-3, to=1-5]
	\arrow["{\rho _{\mathcal{C}}}"', from=3-3, to=3-5]
	\arrow["{\mathbf{R}F}", from=1-5, to=3-5]
	\arrow["{\mathbf{R}F'}", from=1-6, to=3-6]
	\arrow["\simeq", from=1-5, to=1-6]
	\arrow["\simeq"', from=3-5, to=3-6]
\end{tikzcd}\]
\end{proof}

\begin{corollary}
Any square 
\[\begin{tikzcd}
	{\mathcal{S}} && {\mathcal{B}} \\
	& \cong \\
	{\mathcal{A}} && {\mathcal{C}}
	\arrow["U"', from=3-1, to=3-3]
	\arrow["{U'}", from=1-1, to=1-3]
	\arrow["F", from=1-3, to=3-3]
	\arrow["{F'}"', from=1-1, to=3-1]
\end{tikzcd}\]
with $U,U'$ being weak equivalences and $F,F'$ being fibrations is a 2-pullback.
\end{corollary}

\begin{proof}
Let $V:\mathcal{S}\to \mathcal{S}'$ be the connecting map to the 2-pullback of $F$ and $U$. Then by the cancellation property of fibrations and by the 2-for-3 property of weak equivalences we get $V\in W\cap J$-$inj$ hence it is an equivalence. 
\end{proof}

\printbibliography
\end{document}